\documentclass[11pt,reqno]{amsart}

\usepackage[T1]{fontenc}
\usepackage[utf8]{inputenc}
\usepackage{amsmath,amssymb,amsthm,mathrsfs}
\usepackage{tikz}
\usetikzlibrary{arrows.meta,decorations.pathmorphing,positioning,calc,patterns}
\usepackage{enumitem}
\usepackage[margin=1.15in]{geometry}
\usepackage{booktabs}
\usepackage[colorlinks=true,linkcolor=blue!55!black,citecolor=blue!55!black]{hyperref}

\theoremstyle{plain}
\newtheorem{theorem}{Theorem}[section]
\newtheorem{proposition}[theorem]{Proposition}
\newtheorem{lemma}[theorem]{Lemma}
\newtheorem{corollary}[theorem]{Corollary}
\newtheorem{conjecture}[theorem]{Conjecture}

\theoremstyle{definition}
\newtheorem{definition}[theorem]{Definition}
\newtheorem{example}[theorem]{Example}

\newtheorem{notation}[theorem]{Notation}
\theoremstyle{remark}
\newtheorem{remark}[theorem]{Remark}

\newcommand{\FF}{\mathcal F}
\newcommand{\Sing}{\operatorname{Sing}}
\newcommand{\Det}{\operatorname{Det}}
\newcommand{\Tr}{\operatorname{Tr}}
\newcommand{\Tang}{\operatorname{Tang}}
\newcommand{\CS}{\mathrm{CS}}
\newcommand{\BB}{\mathrm{BB}}
\newcommand{\Pic}{\operatorname{Pic}}
\newcommand{\VII}{\mathrm{VII}}
\newcommand{\VIIz}{\mathrm{VII}_0}
\newcommand{\VIIp}{\mathrm{VII}_0^{+}}
\newcommand{\Theet}{\Theta}
\newcommand{\Tors}{\mathrm{Tors}}
\newcommand{\CC}{\mathbb C}
\newcommand{\ZZ}{\mathbb Z}
\newcommand{\PP}{\mathbb P}
\newcommand{\Oo}{\mathcal O}
\newcommand{\Ii}{\mathcal I}

\usetikzlibrary{arrows.meta,calc,decorations.markings,positioning,shapes.geometric}

\definecolor{ccurveA}{RGB}{20,70,160}     
\definecolor{ccurveB}{RGB}{190,30,45}     
\definecolor{csing}{RGB}{240,140,0}       
\definecolor{cshell}{RGB}{60,140,195}     
\definecolor{cleaf}{RGB}{125,125,125}     
\definecolor{cok}{RGB}{0,115,70}          
\definecolor{cbad}{RGB}{190,30,45}        
\definecolor{cbad2}{RGB}{220,150,20}      

\tikzset{
  cv1/.style={ccurveA,line width=1.1pt,line cap=round},
  cv2/.style={ccurveB,line width=1.1pt,line cap=round},
  cvK/.style={black!75,line width=1.1pt,line cap=round},
  leaf/.style={cleaf,line width=0.45pt},
  sing/.style={circle,fill=csing,draw=black,line width=0.35pt,
               inner sep=0pt,minimum size=4.4pt},
  lb/.style={font=\small},
  slb/.style={font=\scriptsize},
  xs/.style={font=\tiny},
  xmark/.style={cbad,font=\bfseries},
  inv1/.style={postaction={decorate,decoration={markings,
      mark=between positions 0.14 and 0.90 step 0.19 with
        {\arrow[ccurveA,line width=1.1pt]{Stealth[length=5pt]}}}}},
  inv2/.style={postaction={decorate,decoration={markings,
      mark=between positions 0.20 and 0.80 step 0.30 with
        {\arrow[ccurveB,line width=1.1pt]{Stealth[length=5pt]}}}}},
  ninv/.style={postaction={decorate,decoration={markings,
      mark=between positions 0.08 and 0.94 step 0.12 with
        {\draw[leaf] (0,-0.26)--(0,0.26);}}}},
}

\newcommand{\CurveC}[1][cv1]{%
  \draw[#1] ( 1.35,-1.35) -- (-0.55, 0.55)
            .. controls (-1.10, 1.10) and (-1.85, 0.55) .. (-1.85,0)
            .. controls (-1.85,-0.55) and (-1.10,-1.10) .. (-0.55,-0.55)
            -- ( 1.35, 1.35);}
\newcommand{\CurveA}[1][cv2]{\draw[#1] (-0.16,1.25) -- (1.85,0.44);}
\newcommand{\CurveDzero}[1][cv1]{%
  \draw[#1] (-0.90,-1.50) .. controls (1.20,-0.50) and (1.20,0.50) .. (-0.90,1.50);}
\newcommand{\CurveDone}[1][cv2]{%
  \draw[#1] (0.90,-1.50) .. controls (-1.20,-0.50) and (-1.20,0.50) .. (0.90,1.50);}

\begin{document}

\title[Class VII surfaces with $b_2=3$ and two foliations are Kato]
{Class VII surfaces with $b_2=3$ and two foliations are Kato}

\author{N. Kurnosov, C. Spicer}
\date{August 2026}

\begin{abstract}
Let $S$ be a minimal compact complex surface of class $\VII$ with $b_2(S)=3$.
We prove that if $S$ carries two distinct singular holomorphic foliations, then $S$
contains a global spherical shell. Moreover, by \cite{DO99,OTZ01} it is an Inoue-Hirzebruch
surface. The proof is based on Teleman's theorem on the existence of a cycle of rational curves and results of Dloussky and is inspired by work of Brunella \cite{Bru11}.
\end{abstract}

\maketitle

\begin{center}
    \textit{Dedicated to Professor Fedor Bogomolov on his eightieth birthday}
\end{center}

\setcounter{tocdepth}{1}
\tableofcontents

\section{Introduction}

By the famous Enriques-Kodaira classification \cite{Kod64,Kod66,Kod68,BHPV} there are ten families of minimal compact
complex surfaces. The least studied is Kodaira's
\emph{class $\VII$}, namely, the surfaces $X$ with
\[
b_1(X)=1,\qquad \kappa(X)=-\infty .
\]
A minimal class $\VII$ surface is called a \emph{class $\VIIz$} surface.  If moreover
$b_2>0$ we write $X\in\VIIp$. Every class $\VII$ surface is the blow-up of a unique
$\VIIz$ surface, so the classification problem is the classification of $\VIIz$ surfaces.

The case $b_2=0$ was studied first by Inoue \cite{Ino74} who classified these surfaces under the
assumption that a holomorphic foliation exists. Later, Bogomolov \cite{Bog76,Bog82} showed that a $\VIIz$ surface with
$b_2=0$ is a Hopf or an Inoue surface. A complete proof of Bogomolov's result was obtained independently by Li-Yau-Zheng \cite{LYZ90,LYZ94}
using Hermitian-Einstein metrics and by Teleman \cite{Tel94} using moduli of vector
bundles. Bogomolov's original approach via a group-theoretic argument has been revised in work of Bogomolov-Buonerba-Kurnosov \cite{BBK19}. We refer to \S\ref{ss:b2zero} for a brief summary of these ideas.

It therefore only remains to understand surfaces in class $\VIIz$ with $b_2>0$.
We recall the following definition.

\begin{definition}[Kato \cite{Kat78}]\label{def:gss-intro}
A \emph{global spherical shell} (GSS) in a compact complex surface $X$ is an open
$\Sigma\subset X$ together with a biholomorphism $\varphi:U\to\Sigma$, where $U$ is a
neighbourhood of the unit sphere $S^3\subset\CC^2\setminus\{0\}$, such that
$X\setminus\Sigma$ is connected. A surface $X\in\VIIp$ containing a GSS is a
\emph{Kato surface}.
\end{definition}

Every known example of a surface in
$\VIIp$ is a Kato surface. This leads to the following conjecture.

\begin{conjecture}[Global Spherical Shell conjecture; Kato \cite{Kat78}]\label{conj:GSS}
Every minimal class $\VII$ surface with $b_2>0$ is a Kato surface.
\end{conjecture}

Kato surfaces can be explicitly constructed from $(\CC^2,0)$ via series of blow-ups
(see \S\ref{ss:kato}). Their deformation theory has been studied
\cite{Dlo84,DK98,OT08}. 

Conjecture \ref{conj:GSS} is a fundamental piece in the classification of complex surfaces. We briefly recall what is known about Conjecture~\ref{conj:GSS} (see also
 Table~\ref{tab:state}).
 Following Teleman \cite{TelGauge,Tel10} we can split
Conjecture~\ref{conj:GSS} into
\begin{description}[leftmargin=2.6em]
\item[(C1)] every $X\in\VIIp$ contains a cycle of rational curves;
\item[(C2)] if $X\in\VIIp$ contains a cycle of rational curves, then $X$ is Kato.
\end{description}

 Teleman proved (C1) for $b_2=1$ \cite{Tel05}.  It follows in this case 
 that since $X$ has a curve, it also has GSS by \cite{Nak84}. Teleman proved (C1) for
$b_2=2$ in \cite{Tel10}, and for $b_2=3$ in \cite{Tel17,Tel18}. His approach  uses holomorphic
models near the reduction loci \cite{Tel15}, which makes it hard to extend for $b_2\ge 4$ as the number of
reduction loci which we need to control grows as $2^{b_2}$.  Other parts of Teleman's proof are the variation formula for the determinant
line bundle \cite{Tel17det}, the positivity results of \cite{Tel06}, Buchdahl's
Nakai-Moishezon criterion \cite{Buc00}, and the bubbling analysis of
Dloussky-Teleman \cite{DT12}.

Towards (C2) Nakamura \cite{Nak90} proved that a surface
with a cycle is a global deformation of blown-up primary Hopf surfaces. Dloussky-Oeljeklaus-Toma \cite{DOT03} made the most significant progress on (C2) as they
proved that $b_2(X)$ rational curves imply a GSS. Earlier Dloussky-Oeljeklaus-Toma \cite{DOT00,DOT01} have proved GSS under the
existence of a global holomorphic vector field. Dloussky has shown GSS in \cite{Dlo06} with the assumption of existence of
a numerically $m$-anticanonical divisor. In \cite{Dlo24} another criterion for GSS has been proposed, namely GSS follows from the existence of
twisted logarithmic $1$-form. This result will be used in the current paper. Some other results include Dloussky-Teleman \cite{DT20} who obtained smoothing results
valid for cusps of rank $\le 11$ and Dloussky \cite{Dlo21} determined the shape of the
maximal divisor of a surface with a cycle. 

Since original works of Bogomolov we know that foliations play a very important role. Brunella \cite{Bru11} proved that a $\VIIz$ surface with
$b_2=2$ which admits a singular holomorphic foliation is a Kato surface. This means that the GSS conjecture at $b_2=2$ reduces to the existence of a foliation. Moreover, as
Dloussky-Oeljeklaus \cite{DO99} proved that every Kato surface carries a foliation, this means that a $\VIIp$ surface without a foliation would be a counterexample to \ref{conj:GSS}. They also proved that Kato surface
 carries at most two foliations, and in Oeljeklaus-Toma-Zaffran \cite{OTZ01} it was proved that exactly two
occur precisely for Inoue-Hirzebruch surfaces. Brunella also characterised Inoue
surfaces \cite{Bru13} and hyperbolic Kato surfaces \cite{Bru14} by automorphic
potentials, and proved that Kato surfaces are locally conformally K\"ahler
\cite{Bru11LCK}. Apostolov-Dloussky \cite{AD16,AD18,AD23} showed this in
locally conformally symplectic and twisted-current language.

There is also progress via analytic methods, Streets-Tian's pluriclosed flow \cite{ST10,ST13} would produce
a curve of non-positive self-intersection on every $\VIIp$ surface if its conjectural
regularity holds. This holds for Chern-Ricci flow by
Tosatti-Weinkove \cite{TW13}.

\begin{table}[t]
\centering
\small
\begin{tabular}{@{}lll@{}}
\toprule
$b_2$ / hypothesis & status & reference \\
\midrule
$0$ & classified: Hopf or Inoue & \cite{Bog76,Bog82,LYZ90,Tel94,BBK19}\\
$1$ & GSS conjecture proved & \cite{Tel05}+\cite{Nak84}, or \cite{Tel05}+\cite{DOT03}\\
$2$ & (C1): a cycle exists & \cite{Tel10}\\
$2$ + a foliation & Kato & \cite{Bru11}\\
$3$ & (C1): a cycle exists & \cite{Tel17,Tel18}\\
\textbf{3 + two foliations} & \textbf{Kato} & \textbf{Theorem~\ref{thm:main}}\\
any $b_2$, $b_2$ curves & Kato & \cite{DOT03}\\
any $b_2$, twisted log $1$-form & Kato & \cite[Thm.~4.2]{Dlo24}\\
any $b_2$, global vector field & Kato & \cite{DOT00,DOT01}\\
any $b_2$, num.\ $m$-anticanonical divisor & Kato & \cite{Dlo06}\\
$b_2\ge 4$ & open, even (C1) & \\
\bottomrule
\end{tabular}
\caption{Progress towards Conjecture~\ref{conj:GSS}.}
\label{tab:state}
\end{table}

As there are no known restrictions on the number of foliations on class $\VIIp$ surface our original idea was to extend Brunella's approach to GSS and try to produce needed rational curves via foliations. This approach seems to be more complicated due to aforementioned lack of Teleman's cycle of curves. However, it does work in the case of $b_2=3$ and two foliations with crucial help of recent criteria due to Dloussky, \cite{Dlo24}.

\begin{theorem}\label{thm:main}
Let $S$ be a minimal compact complex surface of class $\VII$ with $b_2(S)=3$. If $S$
carries two distinct singular holomorphic foliations, then $S$ is a Kato surface.
\end{theorem}

Then by \cite[Thm.~5.5]{DO99} and \cite{OTZ01} we have the following.

\begin{corollary}\label{cor:in-hirz}
    Such $S$ is an Inoue-Hirzebruch
surface (even or half).
\end{corollary}

\begin{corollary}\label{cor:main}
A minimal class $\VII$ surface with $b_2=3$ which is not a Kato surface carries at most
one singular holomorphic foliation.
\end{corollary}

\medskip

We now briefly sketch our argument.
We argue by contradiction and assume that $S\in\VIIz$ with $b_2(S)=3$, carries two distinct
foliations, $\mathcal F_1$ and $\mathcal F_2$, and  $S$ is \emph{not} a Kato surface. 

\smallskip
\emph{(1) We study the curve configurations on $S$ (\S\ref{ss:config}).} If there are three curves in $S$, there would exist
a GSS (Theorem~\ref{thm:DOT}).  If there is an elliptic curve, or  two connected components of the
maximal divisor, there would also exist a GSS
(Theorems~\ref{thm:ratell} and \ref{thm:dlostructure}). The existence of at least one curve is provided by Teleman
Theorem~\ref{thm:telcycle}, which guarantees that $S$ has a cycle. 
We then show that there are three possible curve configurations and describe the shapes of the curves in this configuration with respect to a Donaldson basis.

A Chern-class count (Lemma~\ref{lem:L1})
allows only two shapes for $[N_{\FF_i}]$ with respect to a Donaldson basis (cf. Definition \ref{def:donaldson}) which divides our foliation into two classes, type I and type II (see Definition \ref{def:types}).

\smallskip
\emph{(2) Foliations of Type II (\S\ref{ss:noII}).}  In the type II case there is a unique non-degenerate singular point of the foliation and we show that exactly one curve in the cycle on the surface is invariant by the foliation.  We arrive at a contradiction by considering a calculation of the Baum-Bott residues of the foliation.  This part of the argument does not use the existence of two foliations and indeed the argument is similar to how Brunella proceeds in \cite{Bru11}.

\smallskip
\emph{(3) Foliations of Type I (\S\ref{ss:noIpairs}).} 
In this case the key point is to consider the tangency locus between $\mathcal F_1$ and $\mathcal F_2$.  
In fact, every curve is invariant for both $\mathcal F_1$ and $\mathcal F_2$
and so the tangency divisor is some positive linear combination of the available curve classes on $S$.
A contradiction is then reached by comparing the representations of effective divisors with respect to a Donaldson basis against the representation of $N_{\mathcal F_1}$
and $N_{\mathcal F_2}$ with respect to that same basis.

\smallskip \textbf{Acknowledgment}. We are thankful to Fedor Bogomolov, who brought this problem to our attention. NK thanks Federico Buonerba, who put an enormous effort into writing \cite{BBK19}, and Dima Kaledin, who spent hours reading and discussing it again and again.

\section{Class $\VII$ surfaces and GSS conjecture}
\label{sec:classVII-gss}

In this section we discuss a central object of the paper with main properties and constructions.

\subsection{Definitions and basic invariants}\label{ss:def}

For a surface $X$ we
write $b_i(X)$ for the Betti numbers, $q(X)=h^{0,1}(X)$, $p_g(X)=h^{0,2}(X)$,
$\Omega^1_X$ for the sheaf of holomorphic $1$-forms, $K_X=\Omega^2_X$ for the canonical
bundle, and
\[
\Theet_X:=(\Omega^1_X)^\vee
\]
for the \emph{holomorphic tangent bundle}, i.e.\ the sheaf of holomorphic vector fields.
For a reduced divisor $D$ we write $\Omega^1_X(\log D)$ for the sheaf of $1$-forms with
at worst logarithmic poles along $D$ and $\Theet_X(-\log D)$ for its dual, the sheaf of
vector fields tangent to $D$. For the plurigenera we write $P_m(X)=h^0(K_X^{\otimes m})$, and recall that the Kodaira
dimension is $\kappa(X)=-\infty$ if all $P_m$ vanish and
$\limsup_m \log P_m/\log m$ otherwise.

\begin{definition}\label{def:classVII}
$X$ is of \emph{class $\VII$} if $b_1(X)=1$ and $\kappa(X)=-\infty$. It is of class
$\VIIz$ if in addition it is minimal, i.e.\ contains no smooth rational curve of
self-intersection $-1$. We write $X\in\VIIp$ if $X\in\VIIz$ and $b_2(X)>0$.
\end{definition}

\begin{proposition}[{\cite[Ch.~IV, VI]{BHPV}}, {\cite{Don87}}]\label{prop:invariants}
Let $X\in\VIIz$ with $b_2(X)=n$. Then
\begin{enumerate}[label=\textup{(\roman*)},leftmargin=2.4em]
\item $h^{1,0}(X)=0$, $q(X)=1$, $p_g(X)=0$, $\chi(\Oo_X)=0$;
\item $e(X)=b_2(X)=n$, and by Noether's formula $K_X^2=-n$;
\item the intersection form on $H^2(X,\ZZ)/\Tors$ is negative definite, hence by
Donaldson's theorem 
isometric to $(-1)^{\oplus n}$.
\end{enumerate}
\end{proposition}

We are mainly interested in \ref{prop:invariants}.(iii). It inspires the following 

\begin{definition}[Donaldson basis]\label{def:donaldson}
Let $X \in \VIIz$ with $b_2(X) = n>0$.  There is a basis
$e_0,\dots,e_{n-1}$ of $H^2(X,\ZZ)/\Tors$ with
\[
e_i\cdot e_j=-\delta_{ij},\qquad c_1(K_X)=\sum_{i=0}^{n-1}e_i ,
\]
\end{definition}

The existence of such basis provided by \cite[Theorem 1.8]{Dlo06}.

\begin{notation}\label{not:donaldson}
    The second condition is possible because $c_1(K_X)$ is characteristic for the
intersection form and it satisfies $c_1(K_X)^2=-n$. This  determines the signs of
the $e_i$. After fixing a choice of a Donaldson basis, for any
line bundle $L$ on $X$ we write $[L]=(a_0,\dots,a_{n-1})$ to mean $c_1(L)=\sum a_ie_i$ modulo
torsion. Note that
\[
L^2=-\sum_i a_i^2,\qquad K_X\cdot L=-\sum_i a_i .
\]
For a divisor $D$ we use the notation $[D]$ to denote $[\mathcal{O}_X(D)]$.
\end{notation}

\subsection{The case $b_2=0$}\label{ss:b2zero}

\begin{definition}[Hopf surfaces]\label{def:hopf}
A \emph{contraction} of $\CC^2$ is a germ of biholomorphism $\gamma$ at $0$ fixing $0$
whose eigenvalues at $0$ have modulus $<1$. A \emph{primary Hopf surface} is a quotient
$S_\gamma=(\CC^2\setminus\{0\})/\langle\gamma\rangle$ for such a contraction $\gamma$ on
$\CC^2\setminus\{0\}$; a \emph{secondary} Hopf surface is a free finite quotient of a
primary one. 
\end{definition}

Every primary Hopf surface is diffeomorphic to $S^1\times S^3$, so
$b_1=1$, $b_2=0$, $\pi_1\cong\ZZ$. Hopf surfaces are $\VIIz$ surfaces.

\begin{definition}[Inoue surfaces, \cite{Ino74}]\label{def:inoue}
The Inoue (or Inoue-Bombieri) surfaces are the quotients $S_M$, $S^{+}$, $S^{-}$ of
$\mathbb H\times\CC$ by explicit discrete groups of affine transformations built from a
matrix $M\in SL(3,\ZZ)$ with one real eigenvalue $\alpha>1$ and a pair of conjugate
eigenvalues $\beta,\bar\beta$ with $\alpha|\beta|^2=1$, resp.\ from
$N\in GL(2,\ZZ)$. 
\end{definition}

Inoue surfaces all lie in $\VIIz$, have $b_2=0$, and contain no curves at
all.

The natural question is therefore if there are any other $\VIIz$ surfaces with $b_2=0$. The answer is no.

\begin{theorem}[Bogomolov; Li-Yau-Zheng; Teleman]\label{thm:b2zero}
Every minimal class $\VII$ surface with $b_2=0$ is a Hopf surface or an Inoue surface.
\end{theorem}

This was first observed by Bogomolov in
\cite{Bog76,Bog82}, and the proof was completed in \cite{LYZ90,LYZ94} and in
\cite{Tel94}. The original idea of Bogomolov was to study group-theoretic arguments, and this approach was pursued further in \cite{BBK19}.

\begin{remark}\label{rem:delta}
The ideas of Theorem~\ref{thm:b2zero} do
not extend. For $b_2=0$ either $\Theet_X$ contains a rank-one subsheaf ($\Theet_X$ is
\emph{filtrable}) 
and therefore  a foliation  and Inoue's
conditional classification \cite{Ino74} applies; or $\Theet_X$ is non-filtrable, in which
case it is stable for every Gauduchon metric, hence
Hermitian-Einstein by the Kobayashi-Hitchin correspondence \cite{LT95}. In the second
case one uses that the Bogomolov discriminant vanishes:
\[
\Delta(\Theet_X)=4c_2(X)-c_1(X)^2=4b_2+b_2=5b_2 ,
\]
by Proposition~\ref{prop:invariants}(ii). Hence $b_2=0$ in L\"ubke's inequality
means that $\Theet_X$ is projectively flat \cite{LYZ90,Tel94}. For
$b_2>0$ one has $\Delta(\Theet_X)=5b_2>0$ and there is no rigidity for Hermitian-Einstein metric.
\end{remark}

\subsection{Kato's construction}\label{ss:kato}

We recall now a construction due to Kato \cite{Kat78} in two steps. 

\textbf{Blow-up tower}.
Fix $n\ge 1$ and the unit ball $B\subset\CC^2$. Put $X_0=B$ and, for
$i=1,\dots,n$, choose $p_{i-1}\in X_{i-1}$ lying on
$D_{i-1}:=C_1\cup\dots\cup C_{i-1}$ (with $p_0=0$) and let
$\pi_i:X_i=\mathrm{Bl}_{p_{i-1}}X_{i-1}\to X_{i-1}$ be the blow-up, with exceptional
divisor $C_i$. Write $\pi=\pi_1\circ\dots\circ\pi_n:X_n\to B$ and
$D=\pi^{-1}(0)_{\mathrm{red}}=C_1\cup\dots\cup C_n$.
If each $p_{i-1}$ is a generic point of $C_{i-1}$ the $C_i$ form a chain with
\[
C_1^2=\dots=C_{n-1}^2=-2,\qquad C_n^2=-1 ;
\]
if we choose centers at intersection points of existing curves we produce branch points, and
if we allow the last curve to meet an earlier one then it produces cycles.

\textbf{Kato's gluing}.
The open manifold $X_n\setminus D$ is biholomorphic to $B\setminus\{0\}$. Moreover, if one removes an open ball $B'$
around a point $p_n\in C_n$, then it will leave the compact manifold whose two
boundary components $\partial B$ and $\partial B'$ are $3$-spheres.  Then choose a biholomorphism $\varphi$ between collar
neighbourhoods  $\Sigma_{\mathrm{out}}$ of $\partial B$ and 
$\Sigma_{\mathrm{in}}$ of $\partial B'$ of the two and set
\[
X:=X_n/(\Sigma_{\mathrm{out}}\sim\Sigma_{\mathrm{in}}) .
\]
Then $X$ is a compact complex surface, the image $\Sigma$ of the identified shells is a
GSS by construction, and $X\in\VIIp$ with $b_2(X)=n$. Composition of $\varphi$ and $\pi$ says that 
a Kato surface is exactly given by a
contracting germ of $F:(\CC^2,0)\to(\CC^2,0)$ together with a chosen resolution, and $b_2(X)$ is the number of
blow-ups \cite[\S1]{Dlo84}, \cite{Fav00}, \cite[\S2]{Dlo24}.

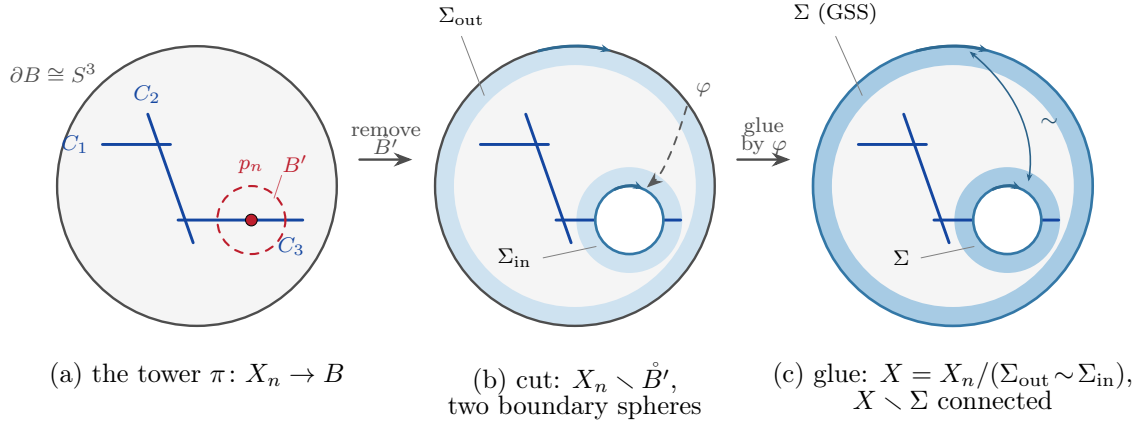
\begin{figure}[htbp]
\centering
\begin{tikzpicture}[>=Stealth,line cap=round]

\begin{scope}[shift={(0,0)}]
  \fill[black!4]                   (0,0) circle (1.85);
  \draw[black!70,line width=0.9pt] (0,0) circle (1.85);
  \node[slb,black!70,anchor=south east,inner sep=1pt] at (-1.28,1.32) {$\partial B\cong S^3$};
  \draw[cv1] (-1.25, 0.55) -- (-0.35, 0.55);
  \draw[cv1] (-0.65, 0.95) -- (-0.05,-0.75);
  \draw[cv1] (-0.25,-0.45) -- ( 1.40,-0.45);
  \node[slb,ccurveA,left]  at (-1.30, 0.55) {$C_1$};
  \node[slb,ccurveA,above] at (-0.66, 0.97) {$C_2$};
  \node[slb,ccurveA,anchor=north] at ( 1.25,-0.54) {$C_3$};
  \draw[ccurveB,dashed,line width=0.8pt] (0.72,-0.45) circle (0.45);
  \node[sing,fill=ccurveB] at (0.72,-0.45) {};
  \node[slb,ccurveB,above] at (0.72,0.04) {$p_n$};
  \draw[ccurveB!70,line width=0.35pt] (1.12,0.14) -- (1.02,-0.16);
  \node[slb,ccurveB,anchor=south west,inner sep=1pt] at (1.09,0.15) {$B'$};
  \node[lb,align=center,anchor=north] at (0,-2.20)
       {(a) the tower $\pi\colon X_n\to B$};
\end{scope}

\begin{scope}[shift={(5.0,0)}]
  \fill[black!4] (0,0) circle (1.85);
  \fill[cshell!25,even odd rule] (0,0) circle (1.85) (0,0) circle (1.60);
  \fill[cshell!25,even odd rule] (0.72,-0.45) circle (0.70) (0.72,-0.45) circle (0.45);
  \draw[cv1] (-1.25, 0.55) -- (-0.35, 0.55);
  \draw[cv1] (-0.65, 0.95) -- (-0.05,-0.75);
  \draw[cv1] (-0.25,-0.45) -- ( 0.27,-0.45);
  \draw[cv1] ( 1.17,-0.45) -- ( 1.40,-0.45);
  \fill[white] (0.72,-0.45) circle (0.45);
  \draw[cshell!85!black,line width=1pt] (0.72,-0.45) circle (0.45);
  \draw[black!70,line width=0.9pt] (0,0) circle (1.85);
  \draw[cshell!75!black,line width=1.1pt,-{Stealth[length=4pt]}]
        (105:1.85) arc (105:75:1.85);
  \draw[cshell!75!black,line width=1.1pt,-{Stealth[length=4pt]}]
        ($(0.72,-0.45)+(115:0.45)$) arc (115:65:0.45);
  \draw[black!65,dashed,line width=0.7pt,->]
        (1.48,1.05) .. controls (1.30,0.45) and (1.15,0.15) .. (0.95,-0.02);
  \node[slb,black!65] at (1.70,1.28) {$\varphi$};
  \draw[black!50,line width=0.35pt] (-1.50,1.98) -- (-1.20,1.20);
  \node[slb,anchor=south] at (-1.50,2.00) {$\Sigma_{\mathrm{out}}$};
  \draw[black!50,line width=0.35pt] (-0.42,-0.95) -- (0.28,-0.80);
  \node[slb,anchor=east] at (-0.44,-0.95) {$\Sigma_{\mathrm{in}}$};
  \node[lb,align=center,anchor=north] at (0,-2.20)
       {(b) cut: $X_n\smallsetminus\mathring{B'}$,\\[-2pt]
        two boundary spheres};
\end{scope}

\begin{scope}[shift={(10.0,0)}]
  \fill[black!4] (0,0) circle (1.85);
  \fill[cshell!45,even odd rule] (0,0) circle (1.85) (0,0) circle (1.60);
  \fill[cshell!45,even odd rule] (0.72,-0.45) circle (0.70) (0.72,-0.45) circle (0.45);
  \draw[cv1] (-1.25, 0.55) -- (-0.35, 0.55);
  \draw[cv1] (-0.65, 0.95) -- (-0.05,-0.75);
  \draw[cv1] (-0.25,-0.45) -- ( 0.27,-0.45);
  \draw[cv1] ( 1.17,-0.45) -- ( 1.40,-0.45);
  \fill[white] (0.72,-0.45) circle (0.45);
  \draw[cshell!85!black,line width=1pt] (0.72,-0.45) circle (0.45);
  \draw[cshell!85!black,line width=1pt] (0,0) circle (1.85);
  \draw[cshell!75!black,line width=1.1pt,-{Stealth[length=4pt]}]
        (105:1.85) arc (105:75:1.85);
  \draw[cshell!75!black,line width=1.1pt,-{Stealth[length=4pt]}]
        ($(0.72,-0.45)+(115:0.45)$) arc (115:65:0.45);
  \draw[cshell!70!black,line width=0.6pt,{Stealth[length=3pt]}-{Stealth[length=3pt]}]
        (0.22,1.78) .. controls (0.80,1.30) and (1.10,0.70) .. (1.00,0.06);
  \node[slb,cshell!55!black,anchor=west,inner sep=1.5pt] at (1.10,0.85) {$\sim$};
  \draw[black!50,line width=0.35pt] (-1.55,1.98) -- (-1.20,1.20);
  \node[slb,anchor=south] at (-1.55,2.00) {$\Sigma$ (GSS)};
  \draw[black!50,line width=0.35pt] (-0.42,-0.95) -- (0.28,-0.80);
  \node[slb,anchor=east] at (-0.44,-0.95) {$\Sigma$};
  \node[lb,align=center,anchor=north] at (0,-2.20)
       {(c) glue: $X=X_n/(\Sigma_{\mathrm{out}}\!\sim\!\Sigma_{\mathrm{in}})$,\\[-2pt]
        $X\smallsetminus\Sigma$ connected};
\end{scope}

\draw[->,line width=0.9pt,black!70] (2.15,0.35) --
     node[slb,above,black!70,align=center,inner sep=1.5pt]{remove\\[-3pt]$\mathring{B'}$} (2.85,0.35);
\draw[->,line width=0.9pt,black!70] (7.15,0.35) --
     node[slb,above,black!70,align=center,inner sep=1.5pt]{glue\\[-3pt]by $\varphi$} (7.85,0.35);
\end{tikzpicture}
\caption{Kato's construction.}
\label{fig:kato}
\end{figure}

\begin{proposition}[{\cite[\S1]{Dlo84}}, {\cite{Kat78}}]\label{prop:katodiff}
A Kato surface obtained from a tower of $n$ blow-ups is diffeomorphic to
$(S^1\times S^3)\,\#\,n\,\overline{\CC\PP^2}$, and contains exactly the $n$ rational
curves $C_1,\dots,C_n$.
\end{proposition}

\subsection{Curves and cycles on class $\VII$ surfaces}\label{ss:curves}

In this subsection we discuss the results about the curves on a class $\VII$ surface. 
First, we recall a bound due to Nakamura on the number of curves on surface in $\VIIp$.

\begin{theorem}[{\cite[\S2]{Nak84}}]\label{thm:curvebound}
Let $S\in\VIIp$ with $b_2(S)=n$. Then $S$ carries at most $n$ irreducible curves. Every
effective divisor on $S$ is a non-negative integral combination of them.
\end{theorem}

\begin{theorem}[{\cite[Lem.~2.2]{Nak84}}, {\cite{Eno81}}]\label{thm:ratell}
Every irreducible curve on $S\in\VIIp$ is rational or elliptic. If $S$ contains an
elliptic curve then $S$ is an Enoki surface, which is a Kato surface.
\end{theorem}

\begin{definition}\label{def:cycle}
A \emph{cycle of rational curves} on a surface $X$ is a reduced connected effective
divisor $\Gamma=D_0+\dots+D_{s-1}$ , which has one of three shapes:
\begin{enumerate}
    \item $s=1$ and $\Gamma=D_0$ is an
irreducible rational curve with a single node;
\item $s=2$ and $D_0,D_1$ are smooth rational
meeting transversally at two distinct points, $D_0\cdot D_1=2$; 
\item $s\ge 3$ and the
$D_i$ are smooth rational with $D_i\cdot D_{i+1}=1$ cyclically and no other
intersections. 
\end{enumerate}In every case $p_a(\Gamma)=1$, i.e.\ $K_X\cdot\Gamma+\Gamma^2=0$. We define
$\sharp(\Gamma):=s$.
\end{definition}

\begin{definition}\label{def:maxdiv}
By Theorem~\ref{thm:curvebound} a surface $S\in\VIIp$ carries finitely many irreducible
curves. The \emph{maximal divisor} $D=D(S)$ is the reduced divisor whose support is
their union.
\end{definition}

Now we can outline main results involving cycles and maximal divisor on $\VIIp$. 
The first is the structural theorem for Kato surfaces, which divides them into classes by the configuration of curves.

\begin{theorem}[{\cite[\S1]{Dlo84}}]
\label{thm:katostructure}
A surface $X\in\VIIp$ is a Kato surface if and only if it arises from
Kato's construction \ref{ss:kato} for some contracting germ. In that
case $X$ carries exactly $b_2(X)$ rational curves, and the dual graph of the maximal
divisor is one of: 
\begin{enumerate}
    \item two disjoint cycles and no trees (\emph{Inoue-Hirzebruch}, or
\emph{hyperbolic});

\item one cycle consisting of $b_2(X)$ curves (\emph{half-Inoue}, or
\emph{odd Inoue-Hirzebruch});

\item one cycle with $\Gamma^2=0$ and no trees (\emph{Enoki}); or
\item one cycle with at least one attached chain (called \emph{intermediate}).
\end{enumerate} 

\end{theorem}

\begin{theorem}[Nakamura, {\cite[(1.5),(1.7)]{Nak90}}]\label{thm:nakamura}
Let $S\in\VIIp$ contain a cycle of rational curves. Then $S$ is a global deformation of a
one-parameter family of blown-up primary Hopf surfaces; in particular
$\pi_1(S)\cong\ZZ$ and $H^2(S,\ZZ)$ is torsion free. 
\end{theorem}

As one important consequence of Theorem \ref{thm:nakamura} we have that
$\Pic^0(S)=\{L_\lambda\}_{\lambda\in\CC^\ast}$ consists of flat line bundles, and for a
line bundle $L$ on $S$ the conditions ``numerically trivial'', ``topologically trivial''
and ``flat'' are equivalent. This follows from $H^1(S,\mathcal{O}_S) \simeq H^1(S, \mathbb C)$ (\cite[p. 444]{Bru11}).

\begin{theorem}[Dloussky, {\cite{Dlo21}}, cf.\ {\cite[Thm.~6.17]{BFR23}}]
\label{thm:dlostructure}
Let $S\in\VIIp$ with $b_2(S)=n$ and let $D$ be the maximal
divisor. Suppose that $S$ contains a cycle $\Gamma$.
\begin{enumerate}[label=\textup{(\alph*)},leftmargin=2.4em]
\item If $D$ has a connected component disjoint from $\Gamma$, that component is itself
a cycle, no trees occur, and $S$ is an Inoue-Hirzebruch surface (hence Kato).
\item Otherwise, $D=\Gamma+A$ is connected, and each connected component of $A$ is a chain
of smooth rational curves meeting $\Gamma$ in exactly one point.  Moreover, distinct chains meet
distinct components of $\Gamma$.
\item $S$ is a half-Inoue surface if and only if $[H_1(S,\ZZ):H_1(\Gamma,\ZZ)]=2$, in
which case $\sharp(\Gamma)=n$ and $\sharp(\Gamma)-\Gamma^2=2n$ (and $S$ is Kato).
Otherwise, $H_1(\Gamma,\ZZ)=H_1(S,\ZZ)$ and
\begin{equation}\label{eq:F8c}
\sharp(\Gamma)-\Gamma^2=n .
\end{equation}
\end{enumerate}
\end{theorem}

\subsection{Explicit Donaldson bases}\label{ss:examples}

Here we consider some explicit examples of computations in Donaldson bases.

\begin{example}\label{ex:hopfbasis}
Let $H$ be a primary Hopf surface and $\widehat H\to H$ the blow-up at $n$ distinct
points, with exceptional curves $\mathcal E_1,\dots,\mathcal E_n$. Since $b_2(H)=0$ we
have $c_1(K_H)=0$ in $H^2(H,\ZZ)/\Tors$, so
$c_1(K_{\widehat H})=\sum_i[\mathcal E_i]$, while
$[\mathcal E_i]\cdot[\mathcal E_j]=-\delta_{ij}$. Thus
\[
e_i:=[\mathcal E_i]
\]
\emph{is} a Donaldson basis: on a blown-up Hopf surface the Donaldson classes are
literally the exceptional classes. By Theorem~\ref{thm:nakamura}, every $X\in\VIIp$
containing a cycle is a global deformation of such an $\widehat H$, and under the
resulting identification of $H^2$ the Donaldson basis of $X$ is the limit of the
exceptional classes. Note $\widehat H\notin\VIIz$: it is not minimal.
\end{example}

\begin{example}[$b_2=1$]\label{ex:b21}
Here $H^2=\ZZ e_0$, $c_1(K)=e_0$. As an easy consequence of the adjunction formula, we have that
$[\Gamma]\in\{0,-e_0\}$:
\begin{itemize}[leftmargin=1.6em]
\item $[\Gamma]=-e_0$, $\Gamma$ a nodal rational curve with $\Gamma^2=-1$: the
half-Inoue surface with $b_2=1$. It carries a nodal rational curve of self-intersection
$-1$ and a foliation \cite{DO99}.
\item $[\Gamma]=0$, $\Gamma$ nodal with $\Gamma^2=0$: an Enoki surface. Here
$\Oo_X(\Gamma)$ is a non-trivial flat line bundle with a section, and \emph{no curve on
$X$ represents $\pm e_0$}. This shows that the curve classes on a $\VIIp$ surface need
neither be linearly independent nor span $H^2$.
\end{itemize}
Both are actually in the same deformation family. Teleman
\cite[\S1]{TelGauge} constructs a holomorphic family $(X_z)_{z\in\Delta}$ of Kato surfaces
with $b_2=1$ whose only curve is homologically trivial and singular rational for $X_z$ with
$z\ne 0$, while the only curve of $X_0$ is singular rational of self-intersection $-1$.
Moreover, the volume of the former tends to infinity as $z\to0$, which is why curves on
non-K\"ahler surfaces cannot be produced by Gromov-Witten-type arguments \cite{TelGauge}.
\end{example}

\begin{example}[$b_2=2$ even Inoue-Hirzebruch]\label{ex:b22ih}
Two disjoint cycles, each necessarily a single nodal rational curve;
$[\Gamma_1]=(-1,0)$ and $[\Gamma_2]=(0,-1)$, both of self-intersection $-1$. Here the
two curve classes do form a Donaldson basis up to sign.
\end{example}

\subsection{The GSS Conjecture}

We now state the main criteria for a surface to be a Kato surface.

\begin{theorem}[Dloussky-Oeljeklaus-Toma]\label{thm:DOT}
Let $S\in\VIIp$ with $b_2(S)=n$. If $S$ has exactly $n$ rational curves, then $S$ contains a global
spherical shell.
\end{theorem}

Combined with Theorem~\ref{thm:katostructure} this gives the following
\begin{center}
    $S$ is Kato $\iff$ $S$ carries $b_2(S)$ rational curves.
\end{center}

 Constructing even one curve is apparently a problem. This has been solved by Teleman for $b_2(S)\le 3$.

\begin{theorem}[Teleman]\label{thm:telcycle}
Let $S\in\VIIp$ with $b_2(S)\le 3$. Then $S$ contains a cycle of rational curves.
\end{theorem}

\begin{theorem}[Teleman, \cite{Tel05}]\label{thm:tel-b_2-1}
    For
$b_2(S)=1$ the GSS conjecture holds.
\end{theorem}

We also have the following complementary criterion proved by Dloussky which will form a key part of our later analysis.

\begin{theorem}[Dloussky, {\cite[Thm.~4.2]{Dlo24}}]\label{thm:dlolog}
Let $S\in\VIIp$. If $S$ carries a non-zero twisted logarithmic $1$-form
\[
\theta\in H^0\big(S,\Omega^1_S(\log E)\otimes L_\lambda\big)
\]
for some non-zero reduced divisor $E$ and some flat line bundle $L_\lambda$, then $S$ is
a Kato surface. Moreover $\lambda$ is real with $\lambda\ge1$, is uniquely determined by
$S$ unless $S$ is Inoue-Hirzebruch, and $\lambda=1$ exactly for Enoki surfaces.
\end{theorem}

This has been also extended to

\begin{theorem}[Apostolov-Dloussky, {\cite[Lem.~6.11]{AD23}}]\label{thm:ADflat}
Let $S\in\VIIz$. If $H^0(S,\Omega^1_S\otimes L)\neq 0$ for some flat line bundle $L$,
then $S$ is a Hopf surface, an Inoue-Bombieri surface, or an Enoki surface. In
particular, if $b_2(S)>0$ then $S$ is a Kato surface.
\end{theorem}

\section{Foliations on class $\VII$ surfaces}\label{sec:fol}

In this section we define the necessary objects for studying foliations on surfaces $\VIIp$, including the Camacho-Sad  and Baum-Bott indices. Both of these notions play a key role in Brunella's proof \cite{Bru11}.

\subsection{Definitions}

We refer to \cite{Bru00} for basic definitions relating to foliations on complex surfaces, but for the reader's convenience we recall some important definitions here.

\begin{definition}\label{def:foliation}
A \emph{holomorphic foliation} $\FF$ on a compact complex surface $S$ is a
saturated coherent rank-one sub-sheaf $T_\FF\subset\Theet_S$ called the \emph{tangent bundle of $\FF$}.
\end{definition}

We emphasise that we are allowing our foliation $\mathcal F$ to be singular at a finite set of points.  Since all foliations appearing in this paper will be holomorphic we will frequently simply refer to holomorphic foliations as foliations.

Note that since $S$ is a smooth surface and $T_\FF$ is reflexive of rank one, $T_\FF$ is a
line bundle. There is an exact sequence
\begin{equation}\label{eq:folseq}
0\longrightarrow T_\FF\longrightarrow \Theet_S\longrightarrow
\Ii_Z\otimes N_\FF\longrightarrow 0
\end{equation}
where $Z=\Sing(\FF)$ is a finite subscheme, the \emph{singular scheme} of $\FF$, and
$N_\FF$ is a line bundle, called the \emph{normal bundle} of $\FF$.

If we take determinants in \eqref{eq:folseq}, we have
\begin{equation}\label{eq:TN}
T_\FF\otimes N_\FF=\det\Theet_S=K_S^{-1}\qquad {\rm equivalently} \qquad
T_\FF=K_S^{-1}\otimes N_\FF^{-1}.
\end{equation}
We define the canonical bundle of $\FF$ to be $K_\FF:=T_\FF^{-1}$.

We define $\ell(S): = $ the number of singular holomorphic foliations on $S$.

\begin{remark}\label{rem:foliation-vector}
Any non-zero rank-one subsheaf $\mathcal L\subset\Theet_S$ gives  a foliation, so
\[
S\ \text{carries a foliation}\iff \Theet_S\ \text{is filtrable}.
\]
In particular, $\FF$ can be given  by a section $v\in H^0(S,\Theet_S\otimes K_{\mathcal F})
= H^0(S,\Theet_S\otimes K_S\otimes N_\FF)$ with isolated zeros.

Locally $\FF$ is generated by a vector field $v_i$ with isolated zeros on $U_i$, with
$v_i=g_{ij}v_j$ for nowhere-vanishing $g_{ij}$.
\end{remark}

\begin{definition}\label{def:invariant}
An irreducible curve $C\subset S$ is \emph{$\FF$-invariant} if for
every local generator $v$ of $T_\FF$ and every local  equation $f$ of $C$ one has
$v(f)\in(f)$.
\end{definition}

\begin{lemma}\label{lem:local}
Let $S$ be a surface and let $\mathcal F$ be a foliation on $S$. Let $C$ be a nodal $\FF$-invariant curve and let $\nu:\widetilde C\to C\subset S$ be its normalization.
\begin{enumerate}[label=\textup{(\roman*)},leftmargin=2.4em]
\item Each local branch of $C$ is $\FF$-invariant, and any local generator $v$ of $T_\FF$
vanishes at every singular point of $C$.  In particular, $\Sing(C)\subset\Sing(\FF)$.
\item If $C'\neq C$ is another $\FF$-invariant curve then $C\cap C'\subset\Sing(\FF)$.
\item 
The morphism $\nu^*T_{\mathcal F} \to \nu^*T_X$ factors through $d\nu\colon T_{\widetilde C} \to \nu^*T_X$, in particular we have a non-zero morphism 
$\nu^*T_{\mathcal F} \to T_{\widetilde C}$ corresponding to a non-zero holomorphic section $\widetilde v \in H^0(\widetilde C, T_{\widetilde C}\otimes\nu^\ast T_\FF^{-1})$.  Moreover, $\widetilde v(q)=0$ for every
$q\in\nu^{-1}(\Sing\FF)$.
\end{enumerate}
\end{lemma}

\begin{proof}
(i) Consider a node, and choose coordinates with $f=zw$, so we write
$v=\alpha\partial_z+\beta\partial_w$. Then $v(f)=\alpha w+\beta z\in(zw)$. When we restrict to
$w=0$, it gives $\beta(z,0)\,z\equiv0$, so $w\mid\beta$; symmetrically $z\mid\alpha$. Hence
$v=z\alpha'\partial_z+w\beta'\partial_w$: each branch is invariant, and $v(0)=0$. The
same conclusion holds at any singular point $p$ of $C$: if $v(p)\neq0$ then in suitable
coordinates $v=\partial_z$, and $v(f)\in(f)$ forces $f$ to be, up to a unit, a function
of $w$ alone, hence smooth at $p$.

(ii) If $v(p)\neq0$ then $\FF$ is regular near $p$ with a unique leaf through $p$, as both
$C$ and $C'$ would lie in it, so $C=C'$ near $p$.

(iii) Let $v \in H^0(X, T_X \otimes K_{\mathcal F})$ be a section defining $\mathcal F$. On $C \setminus \Sing(\FF)$, $v|_C$ gives a section of
$T_C\otimes T_\FF^{-1}$. By (i) this section extends across the preimages of the singular points of
$C$, and an easy local calculation shows that it vanishes on the pre-images of the points in $\Sing(\FF)$.
\end{proof}

\subsection{Local invariants at a singular point}\label{ss:localinv}
Throughout this section let $\mathcal F$ be a foliation on a smooth surface $S$, let $p\in\Sing(\FF)$ and $v=v_1\partial_{z_1}+v_2\partial_{z_2}$ is a local
generator of $T_\FF$ near $p$, with an isolated zero at $p$. 

We will associate to $p$ three basic numerical invariants.

\begin{definition}(Milnor number)
For a finite subscheme $Z\subset S$ with ideal sheaf $\Ii_Z$ one puts
\[
\operatorname{length}(Z):=\sum_{p\in Z}\dim_\CC\big(\Oo_{S,p}/\Ii_{Z,p}\big).
\]
The singular scheme of $\FF$ is $\Sing(\FF)=V(v_1,v_2)$ locally, and the
\emph{multiplicity} (or \emph{Milnor number}) of $\FF$ at $p$ is
\[
\mu(\FF,p):=\dim_\CC\big(\Oo_{S,p}/(v_1,v_2)\big)<\infty ,
\]
so that $\operatorname{length}\Sing(\FF)=\sum_p\mu(\FF,p)$. Note that $\mu(\FF, p)$ is independent of the choice of
generator. A point $p \in \Sing(\FF)$ is
\emph{non-degenerate} if $\det Dv(p)\neq 0$, where $Dv=(\partial v_i/\partial z_j)$;
equivalently, if both eigenvalues $\lambda_1,\lambda_2$ of $Dv(p)$ are non-zero.
\end{definition}

\begin{example}\label{ex:mult}
\begin{enumerate}[label=\textup{(\roman*)},leftmargin=2.4em]

\item The \emph{saddle-node} $v=z(1+\nu w)\partial_z+w^2\partial_w$, $\nu\in\CC$: here
$(v_1,v_2)=(z,w^2)$ because $1+\nu w$ is a unit, and $\Oo_{S,p}/(z,w^2)$ has
$\CC$-basis $\{1,w\}$, so $\mu(\FF,p)=2$ at the single point $p$. Its
eigenvalues are $1$ and $0$ and is therefore degenerate. 

\item The \emph{radial} vector field $v=z\partial_z+w\partial_w$ is non-degenerate.
\end{enumerate}
\end{example}

\begin{definition}\label{def:CS}(Camacho-Sad index)
Let $p \in S$ be a germ of a surface and let $\mathcal F$ be a foliation on $S$.
Let $B$ be a germ of an $\FF$-invariant curve passing through $p$ and let $\{f = 0\}$ be a local equation for $B$ near $p$.  Let $\omega$ be a local $1$-form which defines $\mathcal F$ near $p$.  Since $B$ is $\mathcal F$ invariant we may find holomorphic functions $g, h \in \mathcal O_{S, p}$ which are co-prime to $f$ and a $1$-form $\eta$ such that 
\[g\omega = hdf+f\eta.\] We then define
\[\CS(\mathcal F, B, p) = {\rm Res}_p -\frac{1}{h}\eta|_B = -\frac{1}{2\pi i} \oint_\gamma \frac{1}{h}\eta\]
where $\gamma \subset B$ is a union of small circles around $p$ for each component of $B$ near $p$.
\end{definition}

\begin{remark}\label{rem:CS}
    This is independent of the coordinates and of the generator \cite{CS82},
\cite[Ch.~3]{Bru00}, \cite[\S IV]{Suw98}. We can treat the Camacho-Sad index as "self-intersection
contributed at $p$" by the leaf along $B$. 
\end{remark}

\begin{example}\label{ex:CS}
We compute the Camacho-Sad index in some important examples, cf. \cite[Ch.~3]{Bru00}.

\begin{enumerate}[label=\textup{(\roman*)},leftmargin=2.4em]
\item (Non-degenerate with two invariant curves) Let $v=\lambda_1z\alpha(z, w)\partial_z+\lambda_2w\beta(z, w)\partial_w$ be a local generator of $\mathcal F$ where $\alpha, \beta$ are holomorphic and $\alpha(0, 0) = \beta(0, 0)=1$.  Set $p = (0, 0)$.
In this case
$h = z\alpha(z, w)$, $\eta=-\frac{\lambda_2}{\lambda_1}\beta(z, w)dz$, so
\[\CS(\FF,\{w = 0\}, p)=\operatorname{Res}_{z=0}\frac{\lambda_2\beta(z, 0)}{\lambda_1z\alpha(z, 0)}=\frac{\lambda_2}{\lambda_1}.\]
Exchanging $z$ and $w$ gives $\CS(\FF,\{z=0\},p)=\frac{\lambda_1}{\lambda_2}$, i.e., the two Camacho-Sad indices
are reciprocal.  A similar calculation shows that 
\[\CS(\FF,\{zw = 0\}, p) = \frac{\lambda_1}{\lambda_2}+\frac{\lambda_2}{\lambda_1}+2.\]
\item (Saddle-node) Let $v=z(1+\nu w)\partial_z+w^2\partial_w$ be a local generator of $\mathcal F$ and let $p = (0, 0)$. Along the 
separatrix $\{w=0\}$ we have $h=z(1+\nu w)$, $\eta=-wdz$ giving \[\CS(\FF,\{w=0\},p)=\operatorname{Res}_{z=0} 0 = 0.\] Along the 
separatrix $\{z=0\}$ we have 
\[\CS(\FF,\{z=0\},p)=\operatorname{Res}_{w=0}\frac{1+\nu w}{w^2}=\nu.\]
As above, we see that \[\CS(\FF,\{zw=0\},p)=\nu+2.\]
\end{enumerate}
\end{example}

\begin{theorem}[Camacho-Sad formula, {\cite{CS82}, \cite[Theorem 3.2]{Bru00}}]\label{thm:CSglobal}
Let $\FF$ be a singular holomorphic foliation on a surface $S$ and $C\subset S$ a
 compact connected $\FF$-invariant curve. Then
\[
C^2=\sum_{p\in\Sing(\FF)\cap C}\CS(\FF,C,p).
\]
\end{theorem}

We emphasise that we do not need to assume that $C$ is smooth in order for 
Theorem \ref{thm:CSglobal} to hold.

\begin{definition}\label{def:BB}(Baum-Bott index)
Let $p \in \Sing(\FF)$ and let $v = v_1\partial_{z_1}+v_2\partial_{z_2}$ be a local generator of $\mathcal F$ near $p$.
Set $\operatorname{tr}Dv=\partial_{z_1}v_1+\partial_{z_2}v_2$, and we define
\[
\BB(\FF,p):=\operatorname{Res}_p
\begin{bmatrix}(\operatorname{tr}Dv)^2\,dz_1\wedge dz_2\\[2pt] v_1,\ v_2\end{bmatrix}
=\frac{1}{(2\pi i)^2}\oint_{|v_1|=|v_2|=\varepsilon}
\frac{(\operatorname{tr}Dv)^2\,dz_1\wedge dz_2}{v_1v_2},
\]
the \emph{Grothendieck residue} of the symmetric function $c_1^2$ evaluated on $Dv$.
\end{definition}

\begin{example}\label{ex:BB}
We compute the Baum-Bott index in some important examples.
\begin{enumerate}[label=\textup{(\roman*)},leftmargin=2.4em]
\item (Non-degenerate) Consider the case where $v = \lambda_1z\partial_z+\lambda_2w\partial_w$ with $\lambda_1\lambda_2 \neq 0$. The transformation
law for Grothendieck residues reduces the pair $(\lambda_1z,\lambda_2w)$ to $(z,w)$ at
the cost of a factor $(\lambda_1\lambda_2)^{-1}$, so
\[
\BB(\FF,p)=\frac{(\lambda_1+\lambda_2)^2}{\lambda_1\lambda_2}=\mu+\mu^{-1}+2,
\qquad \mu:=\lambda_2/\lambda_1 .
\]
For the radial singularity, i.e., $\lambda_1=\lambda_2$, this is $4$.
For a general non-degenerate vector field, if $\mu$ is the ratio of the eigenvalues of the vector field at $p$, a similar calculation shows that 
$\BB(\FF,p) = \mu+\mu^{-1}+2$, cf. \cite[pg. 25]{Bru00}.

\item (Saddle-node) Suppose that $v=z(1+\nu w)\partial_z+w^2\partial_w$.  Here
$\operatorname{tr}Dv=1+(\nu+2)w$, and reducing $(z(1+\nu w),w^2)$ to $(z,w^2)$ multiplies
the numerator by $(1+\nu w)^{-1}$, 

So, we can compute
\begin{align*}
\frac{ ({\rm tr} Dv)^2}{v_1v_2} &= \frac{(1+(\nu+2)w)^2}{z(1+\nu w)w^2}\\ &= \frac{(1+2(\nu+2)w+(\nu+2)^2w^2)(1-\nu w +\nu^2 w^2 + \dots)}{zw^2} \\
&= \frac{1+(\nu+4)w+\dots}{zw^2}.
\end{align*}
Note that $\BB(\FF,p)$ is the coefficient of $w$ in the numerator of the
above expression and so
\[
\BB(\FF,p)=\nu+4.
\]
\end{enumerate}
\end{example}

\begin{theorem}[Baum-Bott, {\cite{BB70}}]\label{thm:BBglobal}
Let $\FF$ be a singular holomorphic foliation with isolated singularities on a compact
surface $S$. Then
\[
\sum_{p\in\Sing(\FF)}\BB(\FF,p)=N_\FF\cdot N_\FF .
\]
\end{theorem}

\begin{remark}\label{rem:BB-Brunella}
In \cite[Lem.~1.2]{Bru11} Brunella miscalculates the Baum-Bott index of a saddle node as $\BB(\mathcal F, p) = 2(\nu+2)$- the two agree only for $\nu=0$. Using the value given in Example \ref{ex:BB} allows us to simplify Brunella's
argument somewhat (we use the notation established in \cite{Bru11}). Indeed, we have that $\Det(\FF)=2$, and so $\Sing(\mathcal F)$ either consists of two non-degenerate points or
one saddle-node. Eliminating the latter possibility is the content of  \cite[Lem.~1.2]{Bru11}. The cycle, $C$, on the surface is nodal,  $C^2=-1$, and the saddle-node at the node, the branch
indices are $0$ and $\nu$ hence by Example~\ref{ex:CS}(ii), we have that
 $\nu=-3$.  So $\Tr(\FF)=\BB(\FF,p)=\nu+4=1$ by
Example~\ref{ex:BB}(ii), but this is an immediate contradiction of the fact that $\Tr(\FF)=N_\FF^2=-\sum a_i^2\le0$.
\end{remark}

\subsection{Foliations on $\VIIp$ surfaces}

Now we study foliations on $\VIIp$ surfaces with $b_2=n>0$.

\begin{lemma}\label{lem:L1}
Let $S$  be a $\VIIp$ surface with $b_2=n>0$.  Let $\mathcal F$ be a foliation on $S$ and write 
$[N_{\mathcal F}] = (a_0, a_1, \dots, a_{n-1})$.
Put
\[
\Det(\FF):=\operatorname{length}\Sing(\FF)=\sum_{p\in\Sing(\FF)}\mu(\FF,p),\qquad
\Tr(\FF):=\sum_{p\in\Sing(\FF)}\BB(\FF,p).
\]
Then
\[
\Det(\FF)=c_2(S)+K_S\cdot N_\FF+N_\FF^2=n-\sum_i a_i(a_i+1),
\qquad
\Tr(\FF)=N_\FF^2=-\sum_i a_i^2 .
\]
Moreover $\sum_i a_i(a_i+1)$ is even, so
\[
\Det(\FF)\equiv n \pmod 2 ,
\]
and $0\le\Det(\FF)\le n$.
\end{lemma}

\begin{proof}
From \eqref{eq:folseq}, $c(\Theet_S)=c(T_\FF)\,c(\Ii_Z\otimes N_\FF)$. The rank-one
torsion-free sheaf $\Ii_Z\otimes N_\FF$  has $c_1(\Ii_Z\otimes N_\FF)=c_1(N_\FF)$ and
$c_2=\operatorname{length}(Z)$, so comparing second Chern classes,
\[
c_2(S)=T_\FF\cdot N_\FF+\operatorname{length}(Z).
\]
By \eqref{eq:TN}, we have $T_\FF=-K_S-N_\FF$, so
$T_\FF\cdot N_\FF=-K_S\cdot N_\FF-N_\FF^2$. This gives the first identity. Substituting
$c_2(S)=n$ (Proposition~\ref{prop:invariants}(ii)), $K_S\cdot N_\FF=-\sum a_i$ and
$N_\FF^2=-\sum a_i^2$ (Notation~\ref{not:donaldson}) yields
\begin{align}
\label{eqdet}
\Det(\FF)=n-\sum a_i-\sum a_i^2 = n-\sum a_i(a_i+1).
\end{align}
The identity $\Tr(\FF)=N_\FF^2$ is
Theorem~\ref{thm:BBglobal}. Each $a_i(a_i+1)$ is a product of consecutive integers,
hence is even and non-negative and so $\sum a_i(a_i+1)\ge0$ is even.  Combining this with \eqref{eqdet} gives $n \ge  \Det(\FF)\ge 0$
where the second inequality clearly holds, since $\Det(\FF)$ is the dimension of some vector space.
\end{proof}

\begin{proposition}\label{prop:mustbesingular}
Let $S\in\VIIp$ and let $\FF$ be a foliation on $S$. Then $\Sing(\FF)\neq\emptyset$.
\end{proposition}

\begin{proof}
Let $n = b_2>0$.  If $\Sing(\FF)=\emptyset$ then Lemma~\ref{lem:L1} implies that 
$\Det(\FF)=\Tr(\FF)=0$.  But $\Tr(\FF)=-\sum a_i^2=0$ forces all $a_i=0$, and so
$0 = \Det(\FF)=n-0=n$, a contradiction.
\end{proof}

\begin{remark}\label{rem:regularfol}
This also follows from Brunella's classification of regular holomorphic foliations on
compact complex surfaces \cite{Bru97}: a non-K\"ahler compact surface carrying a smooth
holomorphic foliation is a Hopf or an Inoue surface, hence has $b_2=0$. 
\end{remark}

Moreover, the singularities sit at intersection points of the curve
configuration.

\begin{lemma}\label{lem:singlocation}
\textup{(\cite[Lem.~2.10]{Dlo24})} Let $S\in\VIIp$ and suppose that $\FF$ is the
foliation associated with a non-zero twisted logarithmic $1$-form
\[
\theta\in H^0\big(S,\Omega^1_S(\log D)\otimes L_\lambda\big).
\]
Then $\FF$ has no singularity on $D$ other than the intersection points of the
components of $D$ (nodes included).
\end{lemma}

\subsection{Type I and Type II foliations}\label{ss:types}

\begin{corollary}\label{cor:detn3}
Let $S \in \VIIp$ and suppose that $b_2 = 3$. Let $\mathcal F$ be a foliation on $S$.  Then $\Det(\FF)\in\{1,3\}$, i.e.\
$\sum_i a_i(a_i+1)\in\{0,2\}$.
\end{corollary}

\begin{proof}
    By \ref{lem:L1} we have that $\Det(\FF)$ is odd. By 
Proposition~\ref{prop:mustbesingular} this gives two possible values of $\Det(\FF)$.
\end{proof}

 Note that $a(a+1)=0$ holds if and only if $a\in\{0,-1\}$, and $a(a+1)=2$ holds if and only if
$a\in\{1,-2\}$, and $a(a+1) \ge 6$ otherwise, Corollary~\ref{cor:detn3} splits the foliations on $S$ into two disjoint classes.

\begin{definition}\label{def:types}
Let $S \in \VIIp$ and suppose that $b_2 = 3$. Let $\mathcal F$ be a foliation on $S$ and let $[N_{\mathcal F}] = (a_0, a_1, a_2)$.  We place $\mathcal F$ into one of two disjoint classes of foliations on $S$ depending on the value of $\Det(\FF)$. We say that
\begin{itemize}[leftmargin=1.8em]
\item $\FF$ is of \textbf{type I} if $\sum_i a_i(a_i+1)=0$, equivalently $\Det(\FF)=3$;
\item $\FF$ is of \textbf{type II} if $\sum_i a_i(a_i+1)=2$, equivalently 
$\Det(\FF)=1$.
\end{itemize}
Every foliation on a $\VIIp$ surface with $b_2=3$ is exactly one of these two types.
\end{definition}

\begin{lemma}\label{lem:nondeg}
Let $S \in \VIIp$ and suppose that $b_2 = 3$. Let $\mathcal F$ be a foliation on $S$.
If $\FF$ is of type II then $\Sing(\FF)$ consists of a single point $p$, and $p$ is
non-degenerate.
\end{lemma}

\begin{proof}
The condition $\Det(\FF)=1$ forces $\Sing(\FF)=\{p\}$ and $\mu(\FF,p)=1$.  It is easy to verify that a singular point of a foliation 
is non-degenerate if and only if $\mu(\FF, p) = 1$ and we can conclude.
\end{proof}

\begin{remark}\label{rem:parity}
Lemma~\ref{lem:L1} gives $\Det(\FF)\equiv n\bmod2$, so for $n$ odd the smallest
admissible value of $\Det$ is $1$, which implies that it is a single non-degenerate singularity. If $n$ is
even it is $2$, realised either by two non-degenerate points \emph{or} by one point of
multiplicity two - a saddle-node, Example~\ref{ex:mult}(ii). At $n=2$ one has
$\Det(\FF)=2$ always, and eliminating the saddle-node is exactly
\cite[Lem.~1.2]{Bru11}. At $n=3$ this case simply does not occur.
\end{remark}

\subsection{Foliations on Kato surfaces}\label{ss:fol-kato}

If we assume that $S$ is a Kato surface then we have restrictions on number of foliations obtained in works of Dloussky-Oeljeklaus and Oeljeklaus-Toma-Zaffran.

\begin{theorem}[{\cite[Lem.~5.4, Thm.~5.5]{DO99}}, {\cite{OTZ01}},
{\cite[Thm.~2.17]{Dlo24}}]\label{thm:count}
Let $X$ be a Kato surface with maximal divisor $D$.
\begin{enumerate}[label=\textup{(\alph*)},leftmargin=2.4em]
\item $X$ carries at least one singular holomorphic foliation.
\item Every foliation on $X$ is defined by a twisted logarithmic $1$-form
$\theta\in H^0(\Omega^1_X(\log D)\otimes L_\lambda)$, uniquely determined by the
foliation up to a multiplicative constant.
\item $\ell(X)\le 2$, and $\ell(X)=2$ if and only if $X$ is an Inoue-Hirzebruch surface
(even or half). Moreover $\lambda\in\mathbb R_{\ge1}$, and $\lambda=1$ iff
$X$ is an Enoki surface.
\end{enumerate}
\end{theorem}

So Kato surfaces carry exactly one or exactly two foliations.

For the case of $b_2=2$ it leaves only three possibilities for the general class $\VIIp$ surface. Moreover as the existence of at least one foliation implies by Brunella that surface is Kato, we have that the Conjecture~\ref{conj:GSS} at $b_2=2$ is equivalent to the statement that
every $\VIIz$ surface with $b_2=2$ carries at least one foliation.

\begin{proposition}\label{prop:tri2}
Let $S\in\VIIz$ with $b_2(S)=2$. Then $\ell(S)\in\{0,1,2\}$ and
\begin{itemize}[leftmargin=1.6em]
\item $\ell(S)=2\iff S$ is an Inoue-Hirzebruch surface (even or half);
\item $\ell(S)=1\iff S$ is an Enoki or an intermediate Kato surface;
\item $\ell(S)=0\iff S$ is not a Kato surface $\iff S$ is a counterexample to
Conjecture~\ref{conj:GSS}.
\end{itemize}

\end{proposition}

The last case by Theorem \ref{thm:count} can be generalised to the following.

\begin{proposition}\label{prop:nofol}
If $S\in\VIIp$ carries no singular holomorphic foliation, then $S$ contains no global
spherical shell. Such an $S$ is a counterexample to Conjecture~\ref{conj:GSS}.
\end{proposition}
\begin{proof}
Follows from Theorem~\ref{thm:count}(a).
\end{proof}

\begin{conjecture}\label{conj:everyfol}
Every $S\in\VIIp$ carries at least one singular holomorphic foliation.
\end{conjecture}

This conjecture, if affirmative, would exclude a possible example from Proposition~\ref{prop:nofol}.

To finish this section we state another conjecture which one can treat as generalization of Theorem~\ref{thm:count}:

\begin{conjecture}\label{conj:count}
    Let $S$ be a class $\VIIp$ surface. Then it has at most two foliations.
\end{conjecture}

\subsection{Tangency divisor}
We begin by defining the tangency divisor between two foliations on a surface.

\begin{definition}\label{def:tangency}(Tangency divisor)
Let $\FF_1, \FF_2$ be distinct singular holomorphic foliations on a surface $S$, generated
 by twisted vector fields $v_i \in H^0(S, \Theet_S \otimes K_{\mathcal F_i})$. Consider the section
\[
v_1\wedge v_2\in H^0\big(S,\Lambda^2\Theet_S\otimes K_{\mathcal F_1}\otimes K_{\mathcal F_2}\big).
\]
Since $\FF_1$ and $\FF_2$ are distinct $v_1\wedge v_2$ is a non-zero section and the vanishing locus of $v_1\wedge v_2$ considered as a section of $\Lambda^2\Theet_S\otimes K_{\mathcal F_1}\otimes K_{\mathcal F_2}$ gives an effective divisor on $S$ called the \emph{tangency divisor}, denoted
$\Tang(\FF_1,\FF_2)$.  
\end{definition}

\begin{remark}
\label{rem_basic_tangency_remarks}
Let us continue to use the notation of Definition  \ref{def:tangency}.
    If $p \in \Sing(\FF_1)$, then $p \in \Tang(\FF_1,\FF_2)$.  
    Indeed, $v_1$, considered as a global section of $\Theta_S \otimes K_{\mathcal F_1}$, vanishes at $p$, and hence $v_1 \wedge v_2$ vanishes at $p$.

    \medskip

    Similarly, if $p \not\in \Sing(\FF_1) \cup \Sing(\FF_2)$, then 
    $p \in \Tang(\FF_1,\FF_2)$ if and only if $T_{\FF_1}(p)$ and $T_{\FF_2}(p)$ are co-linear when considered as subspaces of $\Theta_S(p)$. In particular, if $C$ is a curve which is invariant by both $\FF_1$ and $\FF_2$ it follows that $C \subset \Tang(\FF_1,\FF_2)$.
\end{remark}

We make note of some basic properties of the tangency divisor for foliations on surfaces $S \in \VIIp$.

\begin{lemma}\label{lem:tango}
Let $S \in \VIIp$ and let $\FF_1, \FF_2$ be two distinct foliations. Then the tangency divisor
$\Tang(\FF_1,\FF_2)$  is a non-zero
effective divisor, and
\[
\Oo_S\big(\Tang(\FF_1,\FF_2)\big)\cong K_S\otimes N_1\otimes N_2,
\qquad N_j:=N_{\FF_j}.
\]
Consequently, if $\{D_1,\dots,D_r\}$ is the set of irreducible curves contained in $S$, then there are integers
$m_j\ge0$, not all zero, such that 
\[
K_S+N_1+N_2=\sum_{j=1}^{r} m_j\,D_j\quad\text{in }H^2(S,\ZZ)/\Tors .
\]
\end{lemma}

\begin{proof}
By Proposition~\ref{prop:mustbesingular} $\Sing(\FF_1) \neq \emptyset$ and so by Remark \ref{rem_basic_tangency_remarks}
$\Tang(\FF_1,\FF_2)$ is non-zero.

Let $v_j\in H^0(S,\Theet_S\otimes K_{\mathcal F_j})$ define $\FF_j$. 
By definition $\Tang(\FF_1,\FF_2)$ is the vanishing locus of the non-zero global section 
$v_1\wedge v_2$ of $\Lambda^2\Theet_S\otimes K_{\mathcal F_1}\otimes K_{\mathcal F_2} \cong K_S \otimes N_1 \otimes N_2$
(where the isomorphism follows from \eqref{eq:TN})
and hence $\Oo_S\big(\Tang(\FF_1,\FF_2)\big)\cong K_S\otimes N_1\otimes N_2$.
\end{proof}

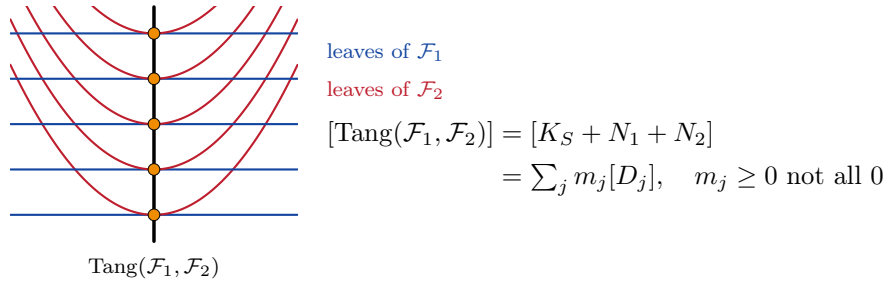
\begin{figure}[htbp]
\centering
\begin{tikzpicture}[line cap=round]
\begin{scope}
  \clip (-1.9,-1.55) rectangle (1.9,1.55);
  \foreach \c in {-1.2,-0.6,0,0.6,1.2}{
     \draw[ccurveA,line width=0.8pt] (-1.9,\c) -- (1.9,\c);
     \draw[ccurveB,line width=0.8pt]
        plot[domain=-1.9:1.9,samples=60,smooth] ({\x},{\c+0.55*\x*\x});}
\end{scope}
\draw[black,line width=1.3pt] (0,-1.55) -- (0,1.55);
\foreach \c in {-1.2,-0.6,0,0.6,1.2}{\node[sing] at (0,\c) {};}
\node[slb,anchor=north] at (0,-1.62) {$\mathrm{Tang}(\mathcal F_1,\mathcal F_2)$};
\node[slb,ccurveA,anchor=west] at (2.15,0.95) {leaves of $\mathcal F_1$};
\node[slb,ccurveB,anchor=west] at (2.15,0.45) {leaves of $\mathcal F_2$};
\node[lb,anchor=west,align=left] at (2.15,-0.45)
     {$[\mathrm{Tang}(\mathcal F_1,\mathcal F_2)]=[K_S+N_1+N_2]$\\[4pt]
      $\phantom{[\mathrm{Tang}(\mathcal F_1,\mathcal F_2)]}
       =\sum_j m_j[D_j],\quad m_j\ge0$ not all $0$};
\end{tikzpicture}
\caption{Two distinct foliations which are tangent along a curve.}
\label{fig:tangency}
\end{figure}

\subsection{Non-invariant and invariant curves}

If a curve $C$ is not invariant by a foliation $\mathcal F$, we can easily produce a lower bound on   $N_\FF\cdot C$. 

\begin{lemma}\label{lem:ninv}
Let $C\subset S$ be an irreducible curve with normalization $\nu:\widetilde C\to C$, and
suppose $C$ is \emph{not} $\FF$-invariant. Then
\[
N_\FF\cdot C\ \ge\ \chi(\widetilde C)=2-2g(\widetilde C).
\]
In particular, if $\widetilde C\cong\PP^1$ (so $C$ is a smooth or nodal rational curve),
then $N_\FF\cdot C\ge2$.
\end{lemma}

\begin{proof}
Let $\omega\in H^0(S,\Omega^1_S\otimes N_\FF)$ define $\FF$ and
consider the composition of morphisms of line bundles on $\widetilde C$
\[
\eta:\quad T_{\widetilde C}\ \xrightarrow{\ d\nu\ }\ \nu^\ast\Theet_S
\ \xrightarrow{\ \nu^\ast\omega\ }\ \nu^\ast N_\FF ,
\]
i.e.\ $\eta\in H^0\big(\widetilde C,\,T_{\widetilde C}^{-1}\otimes\nu^\ast N_\FF\big)$.

Since $C$ is not $\mathcal F$-invariant,  $\eta\neq0$ so
\[
0\le \deg\big(T_{\widetilde C}^{-1}\otimes\nu^\ast N_\FF\big)
=-\chi(\widetilde C)+\deg\nu^\ast N_\FF=-\chi(\widetilde C)+N_\FF\cdot C ,
\]
using $\deg\nu^\ast N_\FF=N_\FF\cdot C$, which holds because $\nu$ is generically injective
with image $C$.
\end{proof}

We will see that on $\VIIp$ surfaces the
candidate classes give $N_\FF\cdot C\le1$, while a non-invariant rational curve would imply
$N_\FF\cdot C\ge2$. So every curve is automatically $\FF$-invariant, this replaces
\cite[Lem.~1.1]{Bru11}.

\begin{definition}(cf. \cite[pg. 15-16]{Bru00})\label{def:Z}
Let $C$ be an irreducible nodal $\FF$-invariant curve with normalization
$\nu:\widetilde C\to C$.  Let $\widetilde v \in H^0(\widetilde C, T_{\widetilde C}\otimes\nu^\ast T_\FF^{-1})$ be the global section given by Lemma~\ref{lem:local}.(iii).
We define $Z(\FF,C):=\operatorname{div}(\widetilde v)\ge0$.
\end{definition}

\begin{lemma}\label{lem:inv}
Let $S$ be a surface, let $\mathcal F$ be a foliation on $S$, let $C$ be a compact irreducible nodal $\mathcal F$-invariant curve and let $\delta$ be the number of nodes of $C$.  Then,
\begin{equation}\label{eq:inv}
N_\FF\cdot C=C^2-2\delta+Z(\FF,C),
\end{equation}
and
\[
Z(\FF,C)\ \ge\ \nu^{-1}(\Sing(\FF)) 
\]
with equality at each non-degenerate singular point of $\FF$ on $C$.
In particular, \[\deg Z(\FF, C) \ge \#\big\{\text{branches of }C\text{ at points of }\Sing(\FF)\cap C\big\}.\]
\end{lemma}

\begin{proof}
First, let us recall that $p_a(C)=g(\widetilde C)+\delta$.
Next, by definition we have
\[
Z(\FF,C)=\deg\big(T_{\widetilde C}\otimes\nu^\ast T_\FF^{-1}\big)
=\chi(\widetilde C)-T_\FF\cdot C .
\]
By \eqref{eq:TN}, $T_\FF\cdot C=-K_S\cdot C-N_\FF\cdot C$, hence
\[
N_\FF\cdot C=Z(\FF,C)-\chi(\widetilde C)-K_S\cdot C .
\]
The adjunction formula implies $K_S\cdot C=2p_a(C)-2-C^2$, and $\chi(\widetilde C)=2-2g(\widetilde C)$,
so
\[
N_\FF\cdot C=Z-\big(2-2g\big)-\big(2p_a-2-C^2\big)=C^2-2\big(p_a-g\big)+Z
=C^2-2\delta+Z .
\]

For the inequality, by
Lemma~\ref{lem:local}(iii), $\widetilde v$ vanishes at every point of $\nu^{-1}(\Sing\FF)$, and $\nu^{-1}(p)$ has one point for each branch of $C$ at
$p$. For the equality statement let $p$ be non-degenerate. If $p$ is a smooth point of
$C=\{w=0\}$, then we may write $v=\alpha\partial_z+w\beta\partial_w$ where $\alpha, \beta \in \mathcal O_{S, p}$ then
$Dv(p)=\left(\begin{smallmatrix}\alpha_z&\alpha_w\\0&\beta\end{smallmatrix}\right)(p)$ is
invertible, so $\alpha_z(p)\neq0$ and $\widetilde v=\alpha(z,0)\partial_z$ vanishes to
order exactly $1$. If $p$ is a node, then by the proof of Lemma~\ref{lem:local}(i),
$v=z\alpha'\partial_z+w\beta'\partial_w$
where $\alpha', \beta' \in \mathcal O_{S, p}$ with $\alpha'(p)=\lambda_1\neq0$ and
$\beta'(p)=\lambda_2\neq0$. If we restrict to $\{w=0\}$ we have $z\alpha'(z,0)\partial_z$, with a
vanishing order exactly $1$, and symmetrically on the other branch.
\end{proof}

\begin{corollary}\label{cor:nodal}
Let $C\subset S$ be a compact irreducible nodal rational $\FF$-invariant curve with a unique node $p$ which is a
non-degenerate singularity of $\mathcal F$ and is the only singularity of $\FF$ on $C$. Then
\[
N_\FF\cdot C=C^2 = \BB(\FF,p).
\]
\end{corollary}

\begin{proof}
Let $B_1 = \{x = 0\}$ and $B_2 = \{y = 0\}$ be the two branches of $C$ at $p$
for some local coordinates $x, y \in \mathcal O_{S, p}$.
Let $v$ be a generator of $\mathcal F$ near $p$.
Since $p$ is a non-degenerate singularity of $\mathcal F$ which leaves both $B_1$ and $B_2$ invariant we have that (up to multiplying $v$ by a unit) 
\[v = xa(x, y)\partial_x+\mu y b(x, y)\partial_y\]
where $\mu \in \mathbb C^*$
and $a, b \in \mathcal O_{S, p}$ with $a(0, 0) = b(0, 0) = 1$.

By Theorem \ref{thm:CSglobal} and Example \ref{ex:CS}.(i) we have $C^2 = \mu+\mu^{-1}+2$.  The equality $N_\FF \cdot C = C^2$ follows from Lemma \ref{lem:inv}.\eqref{eq:inv} after observing that $\deg Z(\FF, C) = 2$.  Finally, we recall that $\BB(\FF,p) = \mu+\mu^{-1}+2$, which establishes all our claimed equalities.
\end{proof}

\subsection{Kato criterion}

Our goal is to reduce the problem to Theorem~\ref{thm:dlolog} or
Theorem~\ref{thm:ADflat}. To do this we emulate Brunella's strategy used in \cite[Lem.~2.1,~2.2]{Bru11} but modifying it with a recent result of Dloussky \cite{Dlo24}.

 If the normal bundle of $\FF$ is numerically equivalent to
$\mathcal{O}_S(E)$ for some invariant divisor $E$ then dividing the defining
$1$-form by the equation of $E$ produces a twisted \emph{logarithmic} $1$-form, and the
surface is Kato.  We need to get this form.

\begin{lemma}\label{lem:kato-cr}
Let $S\in\VIIp$ contain a cycle of rational curves, let $\FF$ be a foliation on $S$, and let $E\ge0$ be a reduced
$\FF$-invariant divisor with normal crossings (possibly allowing $E=0$), with
\[
N_\FF\equiv \mathcal{O}_S(E)
\]
being numerically equivalent.
Then $S$ is a Kato surface.
\end{lemma}

\begin{proof}
Consider $L:=N_\FF\otimes\mathcal{O}_S(-E)$.  By assumption $L$ is numerically trivial, hence flat by
Theorem~\ref{thm:nakamura}. Let $\omega\in H^0(S,\Omega^1_S\otimes N_\FF)$ define $\FF$
and let $\sigma_E\in H^0(S,\mathcal{O}_S(E))$ be the canonical section, so that
$\theta:=\omega/\sigma_E$ is an a priori meromorphic section of $\Omega^1_S\otimes L$
with poles only on $E$.

We claim that in fact $\theta$ is a logarithmic $1$-form, i.e.\[
\theta\in H^0\big(S,\Omega^1_S(\log E)\otimes L\big).
\]
This can be checked locally in a small neighbourhood of any point $x \in S$.
For $x \not\in E$
there is nothing to check. 

If $x \in E$ is a smooth point of $E$, choose local coordinates with
$E=\{z=0\}$, trivialize $N_\FF$, and write $\omega=a\,dz+b\,dw$. 
Since $E$ is invariant, we have
$\omega\wedge dz\in(z)\cdot\Omega_S^2$, which means $z\mid b$. Then we write $b=zb'$,
\[
\omega=z\Big(a\,\frac{dz}{z}+b'\,dw\Big),
\]
and the bracket is a section of $\Omega^1_S(\log E)$.

If $x \in E$ is a normal crossing point, then in some local coordinates we have $E=\{zw=0\}$ and both
branches are invariant by Lemma~\ref{lem:local}(i), hence $z\mid b$ and $w\mid a$. Writing
$a=wa'$, $b=zb'$, we have
\[
\omega=zw\Big(a'\,\frac{dz}{z}+b'\,\frac{dw}{w}\Big),
\]
and so again $\theta$ is a logarithmic $1$-form.

If $E\neq0$, Theorem~\ref{thm:dlolog} gives that $S$ is a Kato surface. If $E=0$ then
$\theta=\omega\in H^0(S,\Omega^1_S\otimes L)$ with $L$ flat, and
Theorem~\ref{thm:ADflat} together with $b_2(S)>0$ gives that $S$ is a Kato surface (an
Enoki surface).
\end{proof}

\section{Main Theorem}\label{sec:theorem}

In this section we are going to prove Theorem \ref{thm:main} in three steps. First, we study possible configurations of curves on non-Kato surfaces with $b_2=3$. Then we show that there are no foliations of Type II, and then that there is no pair of foliations of Type I.

\subsection{Curve configurations on a non-Kato surface with $b_2=3$}
\label{ss:config}

First, we study possible configurations of curves provided we have Teleman's cycle. This follows ideas of Brunella.

\begin{proposition}\label{prop:config}
Let $S \in \VIIp$ with $b_2(S)=3$ and suppose that $S$ is not a Kato
surface. Up to relabelling the Donaldson basis,
the set of irreducible curves of $S$ falls into exactly one of the following cases (see also  Figure \ref{fig:config}).
\smallskip

\begin{center}
\small
\renewcommand{\arraystretch}{1.25}
\begin{tabular}{@{}cllc@{}}
\toprule
 & Curves & Classes & Self Int.\\
\midrule
\textup{(a)} & $C$ nodal rational & $[C]=(-1,-1,0)$ & $-2$\\
\textup{(a$'$)} & $C$ nodal rational, $A$ smooth, $A\cdot C=1$
 & $[C]=(-1,-1,0)$, $[A]=(1,0,-1)$ & $-2,\,-2$\\
\textup{(b)} & $D_0,D_1$ smooth, $D_0\cdot D_1=2$
 & $[D_0]=(0,1,-1)$, $[D_1]=(-1,-1,1)$ & $-2,\,-3$\\
\bottomrule
\end{tabular}
\end{center}
\smallskip

\noindent 
Moreover, 
\begin{enumerate}
    \item in case \textup{(a$'$)} the point $A\cap C$ is a smooth point of $C$ and $C$ and $A$ meet transversally; 
    \item in case \textup{(b)} the two curves meet transversally at two
distinct points; and
\item in all three cases the classes listed are linearly independent, so a
non-zero effective divisor has non-zero class.
\end{enumerate}  
\end{proposition}

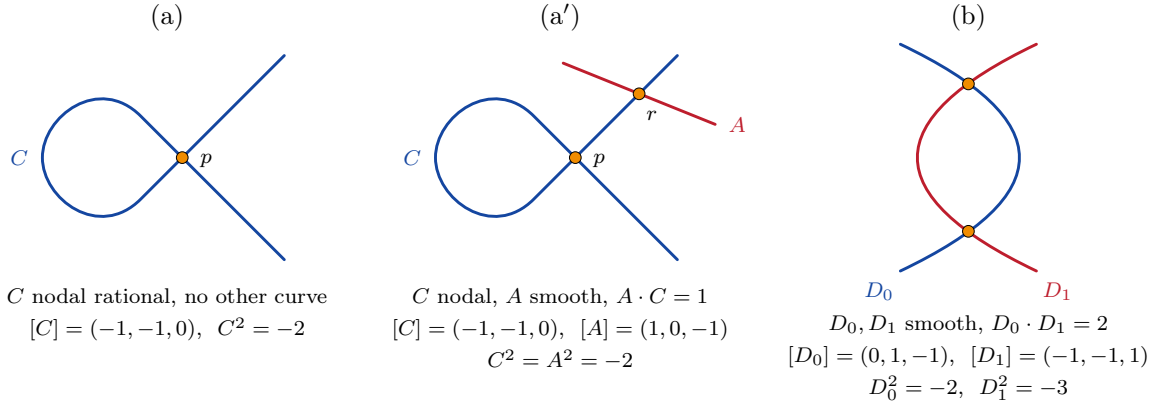
\begin{figure}[htbp]
\centering
\begin{tikzpicture}[line cap=round]

\begin{scope}[shift={(0,0)}]
  \CurveC
  \node[sing] at (0,0) {};
  \node[slb,right=3pt] at (0,-0.04) {$p$};
  \node[slb,ccurveA,left=1pt] at (-1.87,0) {$C$};
  \node[lb,anchor=south] at (-0.2,1.55) {(a)};
  \node[slb,anchor=north,align=center] at (-0.2,-1.60)
       {$C$ nodal rational, no other curve\\[2pt]
        $[C]=(-1,-1,0)$,\ \ $C^2=-2$};
\end{scope}

\begin{scope}[shift={(5.2,0)}]
  \CurveC
  \CurveA
  \node[sing] at (0,0) {};
  \node[slb,right=3pt] at (0,-0.04) {$p$};
  \node[sing] at (0.844,0.844) {};
  \node[slb,anchor=north west,inner sep=1pt] at (0.90,0.66) {$r$};
  \node[slb,ccurveA,left=1pt] at (-1.87,0) {$C$};
  \node[slb,ccurveB,right=1pt] at (1.86,0.44) {$A$};
  \node[lb,anchor=south] at (-0.2,1.55) {(a$'$)};
  \node[slb,anchor=north,align=center] at (-0.2,-1.60)
       {$C$ nodal, $A$ smooth, $A\cdot C=1$\\[2pt]
        $[C]=(-1,-1,0)$,\ \ $[A]=(1,0,-1)$\\[2pt]
        $C^2=A^2=-2$};
\end{scope}

\begin{scope}[shift={(10.4,0)}]
  \CurveDzero
  \CurveDone
  \node[sing] at (0, 0.975) {};
  \node[sing] at (0,-0.975) {};
  \node[slb,ccurveA,anchor=north east] at (-0.85,-1.52) {$D_0$};
  \node[slb,ccurveB,anchor=north west] at (0.85,-1.52) {$D_1$};
  \node[lb,anchor=south] at (0,1.55) {(b)};
  \node[slb,anchor=north,align=center] at (0,-1.95)
       {$D_0,D_1$ smooth, $D_0\cdot D_1=2$\\[2pt]
        $[D_0]=(0,1,-1)$,\ \ $[D_1]=(-1,-1,1)$\\[2pt]
        $D_0^2=-2$,\ \ $D_1^2=-3$};
\end{scope}
\end{tikzpicture}
\caption{The three curve configurations of Proposition~\ref{prop:config}: in (a) and
(a$'$) the cycle is a single nodal rational curve with $\Gamma^2=-2$, in (b)
it is a pair of smooth rational curves meeting twice, $\Gamma^2=-1$.}
\label{fig:config}
\end{figure}

\begin{proof}
Let $D$ denote the maximal divisor on $S$, i.e., the sum of all irreducible curves on $S$.

Since $S$ is not Kato,  Theorem~\ref{thm:ratell} implies each curve in $S$ must be rational.
    Since and $b_2(S) = 3$,  Theorem ~\ref{thm:DOT} implies that $D$ has at most two irreducible components.
   
    By Theorem~\ref{thm:telcycle} $D$ must contain a cycle
 of rational curves.  Let us denote this cycle by $\Gamma$ and write $D = \Gamma+A$. Note that $A$ is either empty or is a single irreducible curve.

Since $S$ is not Kato, Theorem~\ref{thm:dlostructure}.(a) implies that $D$ is connected and Theorem~\ref{thm:dlostructure}.(b) implies that if $A \neq 0$, then it is a smooth rational curve.

We therefore have 3 cases to consider
\begin{itemize}
    \item[(a)] $\Gamma$ has a single irreducible component  and $A = 0$.
    \item[(a$'$)] $\Gamma$ has a single irreducible component  and $A \neq  0$ (note that it is a smooth rational curve).
    \item[(b)] $\Gamma$ has two irreducible components.
\end{itemize}

These three cases  correspond to the three cases listed in the statement of the Theorem and we will verify the remaining claims of the Theorem in each case.

\emph{Case} (a):  In this case it only remains to check that (up to permuting the Donaldson basis) $C^2 = -2$ and $[C] = (-1, -1, 0)$. 
Let $[C] = (a_0, a_1, a_2)$.
By the adjunction formula we know that 
\[(K_S+C)\cdot C = 2p_a(C)-2 = 0.\]    
 Since $[K_S+C] = (a_0+1, a_1+1, a_2+1)$ we can compute easily that 
 \begin{align}
 \label{eq:adjbound}
     a_0(a_0+1)+a_1(a_1+1)+a_2(a_2+1) = 0.
 \end{align}
 Note that $n(n+1)>0$ unless $n \in \{-1, 0\}$, so \eqref{eq:adjbound} implies that $a_0, a_1, a_2 \in \{-1, 0\}$.
By Theorem~\ref{thm:dlostructure}.(c) $S$ is neither Kato nor half-Inoue, and $\Gamma$ has a single irreducible component
Theorem~\ref{thm:dlostructure}.(c).\eqref{eq:F8c} gives us that $1 = 3+\Gamma^2$ and so $\Gamma^2 = -2$, which gives us 
\begin{align}
    \label{eq:selfintbound}
    a_0^2+a_1^2+a_2^2 = 2.
\end{align}
Subtracting \eqref{eq:selfintbound} from \eqref{eq:adjbound} we see that $a_0+a_1+a_2 = -2$.
Together with the fact that $a_0, a_1, a_2 \in \{-1, 0\}$ we see that (up to permuting the Donaldson basis)
$(a_0, a_1, a_2) = (-1, -1, 0)$.

\medskip

\emph{Case} (a$'$):  Arguing exactly as in the previous case we may assume (after permuting the Donaldson basis) that 
$C^2 = -2$ and $[C] = (-1, -1, 0)$.  It therefore remains to show that $A\cdot C = 1$ (equivalently, $A$ and $C$ intersect at a smooth point of $C$), $[A] = (1, 0, -1)$ and $A^2 = -2$.
Let $[A] = (b_0, b_1, b_2)$.  Since $A$ is a smooth rational curve, by the adjunction formula we have 
\[(K_S+A)\cdot A = -2\]
and so
 \begin{align}
 \label{eq:adjbound2}
     b_0(b_0+1)+b_1(b_1+1)+b_2(b_2+1) = 2.
 \end{align}
 Note that $n(n+1) = 0$ if $n \in \{-1, 0\}$, $n(n+1) = 2$ if $n \in \{-2, 1\}$ and otherwise $n(n+1) >2$.
 From this observation we see that \eqref{eq:adjbound2}  holds  only if 
 \begin{align}
 \label{eq:b0b1b2}
  \text{one of $b_0, b_1, b_2$ is $\in \{-2, 1\}$
 and the other two are $\in \{-1, 0\}$. }
 \end{align}

Since $A\cdot C >0$ we also have
\begin{align}
\label{eq:intbound2}
    b_0+b_1 >0.
\end{align}
Since $b_i \in \{-2, -1, 0, 1\}$, \eqref{eq:intbound2} ensures that at least one of $b_0$ or $b_1$ is $=1$.  Up to permuting the first two elements of the Donaldson basis (which we are free to do because our representation of $[C]$ is unchanged by this permutation) we may assume that $b_0 = 1$, and so either $(b_0, b_1) = (1, 1)$ or $=(1, 0)$.
Since \eqref{eq:b0b1b2} implies that $b_1 \in \{-1, 0\}$ we deduce that $b_1 = 0$.  There are therefore two possibilities for $[A]$, either $[A] = (1, 0, 0)$ or $[A] = (1, 0, -1)$.  Suppose for sake of contradiction that $[A] = (1, 0, 0)$. In this case, $A^2 = -1$ and so $A$ is a $(-1)$-curve, which contradicts our assumption that $S$ is minimal.

Thus, $[A] = (1, 0, -1)$ and the claims that $A^2 = -2$ and $A\cdot C = 1$ can be easily verified.

\medskip

\emph{Case} (b):  In this case, it only remains to check that (up to switching $D_0$ and $D_1$ and permuting the Donaldson basis) that  $[D_0] = (0, 1, -1), [D_1] = (-1, -1, 1), D_0^2 = -2$ and $D_1^2 = -3$.
Let $[\Gamma] = [D_0]+[D_1] = (\gamma_0, \gamma_1, \gamma_2)$.
As in case (a), from the adjunction formula we have $(K_S+\Gamma)\cdot \Gamma = 0$ and so $\gamma_0, \gamma_1, \gamma_2 \in \{-1, 0\}$.
We again apply 
Theorem~\ref{thm:dlostructure}.(c).\eqref{eq:F8c} to deduce that $\Gamma^2 =  -1$ and hence 
\[\gamma_0^2+\gamma_1^2+\gamma_2^2 = 1,\]
which, up to a permutation of the Donaldson basis, allows us to  assume that $[\Gamma] = (-1, 0, 0)$.

Let $[D_0] = (c_0, c_1, c_2)$. Since $[\Gamma] = (-1, 0, 0)$ we have $[D_1] = (-1-c_0, -c_1, -c_2)$.  
Since $D_0$ is a smooth rational curve, as in case (a$'$) we may use the adjunction formula $(K_S+D_0)\cdot D_0 = -2$
to see that 
 \begin{align}
 \label{eq:adjbound3}
     c_0(c_0+1)+c_1(c_1+1)+c_2(c_2+1) = 2
 \end{align}
and that 
 \begin{align}
 \label{eq:c0c1c2}
  \text{one of $c_0, c_1, c_2$ is $\in \{-2, 1\}$
 and the other two are $\in \{-1, 0\}$. }
 \end{align}

From $D_0\cdot D_1 = 2$ we have 
that
\begin{align}
\label{eq:intbound3}
    c_0(-1-c_0)-c_1^2-c_2^2 = -2.
\end{align}
Suppose for sake of contradiction that $c_0 =-2$ (resp. $=1$).  In either case, \eqref{eq:intbound3} guarantees that $c_1 = c_2 = 0$.
From this we see that $[D_1] = (1, 0, 0)$ (resp. $[D_0] = (1, 0, 0)$) and hence $D_1^2 = -1$ (resp. $D_0^2 = -1)$, which implies that $D_1$ (resp. $D_0$) is a $(-1)$-curve, a contradiction of our assumption that $S$ is minimal.  

So we see that $c_0 \in \{0, -1\}$ and \eqref{eq:intbound3} then guarantees that $c_1, c_2 \in \{-1, 1\}$.
In fact, \eqref{eq:c0c1c2} shows that $\{c_1, c_2\} = \{-1, 1\}$ so up to permuting the second and  third elements of our Donaldson basis we may assume that $c_1 = 1$ and $c_2 = -1$.
This gives us two possibilities for $[D_0]$ and $[D_1]$ depending on whether $c_0 = 0$ or $c_0 = -1$, namely, 
we have either 
\begin{itemize}
    \item $[D_0] = (0, 1, -1)$ and $[D_1] = (-1, -1, 1)$; or
    \item $[D_0] = (-1, 1, -1)$ and $[D_1] = (0, -1, 1)$.
\end{itemize} 
Up to permuting the second and third coordinates of our Donaldson basis, and swapping the roles of $D_0$ and $D_1$ we may reduce the latter case to the former case and so we may assume that  $[D_0] = (0, 1, -1)$ and $[D_1] = (-1, -1, 1)$ and it is easy to see that $D_0^2 = -2$ and $D_1^2 = -3$.
\end{proof}

\subsection{Type II foliations}\label{ss:noII}

We first rule out the existence of Type II foliations on class $\VIIp$ surfaces which are not Kato.

\begin{proposition}\label{prop:noII}
Let $S \in \VIIp$ with $b_2(S) = 3$ and suppose that $S$ is not a Kato surface.  Then $S$ carries no foliation of type
\textup{II}.
\end{proposition}

\begin{proof}
Suppose for sake of contradiction that there exists a foliation $\FF$ on $S$ of type II.  Let  $[N_\FF]=(a_0,a_1,a_2)$ and (see  Definition~\ref{def:types}) we have that exactly
one $a_i$ lies in $\{1,-2\}$ and the other two are in $\{0,-1\}$, and $\Det(\FF)=1$. By
Lemma~\ref{lem:nondeg}, $\Sing(\FF)=\{p\}$ with $p$ non-degenerate.

By Proposition \ref{prop:config} there are three possible cases for the set of irreducible curves on $S$.  We will argue based on these cases, and use the notation established in Proposition \ref{prop:config} for these curves.

\smallskip
\emph{Cases} (a) \emph{and} (a$'$). Since $[C]=(-1,-1,0)$,
\[
N_\FF\cdot C=-\big(a_0(-1)+a_1(-1)+a_2\cdot0\big)=a_0+a_1 .
\]
At most one $a_i$ equals $1$ and the rest are $\le0$, so $a_0+a_1\le1<2$. As
$\widetilde C\cong\PP^1$, Lemma~\ref{lem:ninv} forces $C$ to be $\FF$-invariant.

By Lemma~\ref{lem:local}(i) the node of $C$ belongs to $\Sing(\FF)=\{p\}$, so $p$ is that
node and it is the only singularity of $\FF$ on $C$. Corollary~\ref{cor:nodal} therefore implies that 
\[
N_\FF\cdot C=C^2=\BB(\FF,p)=-2
\]
and the Baum-Bott formula then implies that 
\[N_\FF^2 = \BB(\FF,p)=-2.\]
By Lemma~\ref{lem:L1}, $\Tr(\FF)=-\sum_{i = 0}^2a_i^2$. Hence
\begin{align}
\label{eq4.2.1}a_0+a_1=-2 \\
\label{eq4.2.2}a_0^2+a_1^2+a_2^2=2.
\end{align}
Since the $a_i$ are all integers, \eqref{eq4.2.2} implies that $a_0, a_1, a_2 \in \{-1, 0, 1\}$.  \eqref{eq4.2.1} then implies that $a_0 = a_1 = -1$. As noted above, exactly one amongst $a_0, a_1, a_2$ lies in $\{1, -2\}$ and therefore $a_2 = 1$.  Thus, $a_0^2+a_1^2+a_2^2 = 3$, which contradicts
\eqref{eq4.2.2}.

\smallskip
\emph{Case} (b). 
Here $[D_0]=(0,1,-1)$, $[D_1]=(-1,-1,1)$, $[\Gamma]=(-1,0,0)$, so
\[
N_\FF\cdot D_0=-a_1+a_2,\qquad N_\FF\cdot D_1=a_0+a_1-a_2,\qquad
N_\FF\cdot D_0+N_\FF\cdot D_1=a_0=N_\FF\cdot\Gamma .
\]
Assume that both $D_0$ and $D_1$ are non-invariant, then Lemma~\ref{lem:ninv} implies that
$N_\FF\cdot D_j\ge2$ for both, hence $a_0\ge4$. However, $a_0\le1$. So at least one of the curves is
invariant. If \emph{both} were invariant, then by Lemma~\ref{lem:local}.(ii) both points
of $D_0\cap D_1$ would lie in $\Sing(\FF)$, which contradicts $\Det(\FF)=1$. Hence exactly
one is invariant. Denote the invariant one by $D$.

Then as $D$ is smooth, the Camacho-Sad formula implies that 
\begin{align}
\label{cs_formula_D}
    D^2=\sum_{q \in \Sing(\FF)\cap D}\CS(\FF,D,q).
\end{align}
Since $\Sing(\FF) = \{p\}$ we therefore have two cases to consider, $p \in D$ and $p \not\in D$.
\begin{itemize}[leftmargin=1.6em]
\item Suppose that $p \not\in D$ and so the sum in the RHS of \eqref{cs_formula_D} is empty, hence $D^2=0$. However, $D_0^2=-2$ and $D_1^2=-3$
and in either case we have our desired contradiction. 
\item Suppose that $p\in D$.  Since $p$ is a non-degenerate singularity of $\mathcal F$, if $\mu$ is the ratio of the eigenvalues of a local generator at $p$, then $\mu$ is well defined and non-zero and we have $\BB(\FF,p)=\mu+\mu^{-1}+2$, see Example \ref{ex:BB}.  
By Lemma ~\ref{lem:L1} \[\BB(\FF, p) =N_\FF^2=-\sum_{i = 0}^2 a_i^2\] must be a negative integer, and hence 
$\mu+\mu^{-1}$ is a negative integer.  Thus, $\mu$ is not a positive rational number and therefore $p$ is a reduced singularity (we refer to \cite[Chapter 1 \S 1]{Bru00} for the notion of a reduced singularity and basic properties), in particular, there are local coordinates $(z, w)$ about $p$ such that $D = \{w = 0\}$ and $\mathcal F$ is generated by the vector field $z\partial_z+\mu w\partial_w$.
By Example \ref{ex:CS} we have $\CS(\mathcal F, D, p)= \mu$.
By \eqref{cs_formula_D} we see that $\mu \in\{-2,-3\}$ and so 
\[
\Tr(\FF)=\BB(\FF,p)=\mu+\mu^{-1}+2=
\begin{cases}
-2-\tfrac12+2=-\tfrac12, & {\rm if } ~D^2=-2,\\[2pt]
-3-\tfrac13+2=-\tfrac43, & {\rm if } ~ D^2=-3 .
\end{cases}
\]
In either case, $\Tr(\FF)$ is not an integer, which gives our sought after contradiction.
\end{itemize}

\begin{figure}[htbp]
\centering
\begin{tikzpicture}[line cap=round]

\begin{scope}[shift={(0,0)}]
  \CurveDzero[cv1,inv1]
  \CurveDone[cv2,inv2]
  \node[sing] at (0, 0.975) {};
  \node[sing] at (0,-0.975) {};
  \node[xmark] at (1.38,1.40) {\LARGE$\times$};
  \node[slb,anchor=north,align=center] at (0,-1.75)
       {(i) both invariant};
\end{scope}

\begin{scope}[shift={(5.3,0)}]
  \CurveDzero[cv1,ninv]
  \CurveDone[cv2,ninv]
  \node[xmark] at (1.38,1.40) {\LARGE$\times$};
  \node[slb,anchor=north,align=center] at (0,-1.75)
       {(ii) both non-invariant};
\end{scope}

\begin{scope}[shift={(10.6,0)}]
  \CurveDzero[cv1,inv1]
  \CurveDone[cv2,ninv]
  \node[sing] at (0.533,0.451) {};
  \node[slb,anchor=west,inner sep=3pt] at (0.60,0.47) {$p$};
  \node[xmark] at (1.38,1.40) {\LARGE$\times$};
  \node[slb,anchor=north,align=center] at (0,-1.75)
       {(iii) exactly one invariant};
\end{scope}
\end{tikzpicture}
\caption{Elimination of Type II foliations in configuration (b). Arrows on a curve mean that it is a leaf, transverse ticks mean that the foliation crosses it.}
\label{fig:typeIIcaseb}
\end{figure}
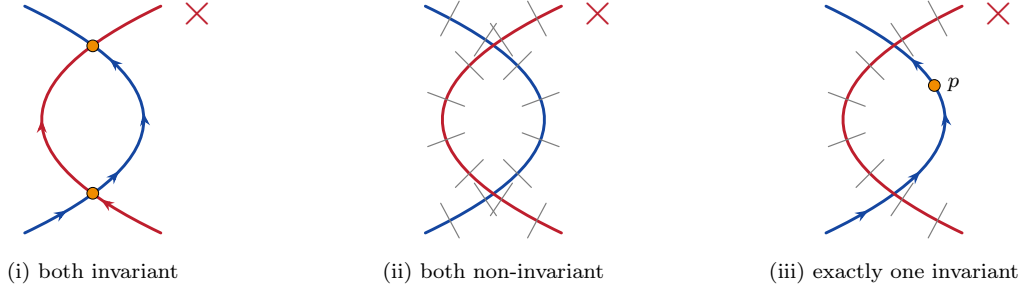
\end{proof}

\subsection{Type I foliations}\label{ss:noIpairs}

By Section \ref{ss:noII} any pair of foliations on our surface must be Type I. Now we show that this is also impossible.

\begin{proposition}\label{prop:noIpairs}
Let $S \in \VIIp$ with $b_2(S) = 3$ and suppose that $S$ is not a Kato surface.  Then $S$ carries at most one singular
holomorphic foliation.
\end{proposition}

\begin{proof}
Suppose for sake of contradiction that one can find two distinct foliations $\mathcal F_1, \mathcal F_2$ on $S$. By Proposition~\ref{prop:noII} both are
of type I, so if we write $[N_{\mathcal F_j}] = (a_0^{(j)}, a_1^{(j)}, a_2^{(j)})$ (cf. Notation~\ref{not:donaldson})
\[
a_i^{(j)}\in\{0,-1\}\quad (i=0,1,2;\ j=1,2).
\]
By Lemma~\ref{lem:tango} $\Tang(\FF_1,\FF_2)=\sum_j m_j\Sigma_j$ is a non-zero effective divisor and if we set $A_i:=a_i^{(1)}+a_i^{(2)}\in\{0,-1,-2\}$ and $t_i:=1+A_i$ we have
\begin{equation}\label{eq:tang-class}
(t_0,t_1,t_2)=[K_S+N_1+N_2]=[\Tang(\FF_1,\FF_2)]=\sum_j m_j[\Sigma_j].
\end{equation}
where $\Sigma_j \subset S$ are irreducible divisors.  
Note also that $(t_0, t_1, t_2) \neq (0, 0, 0)$ by  Proposition~\ref{prop:config}.(3).

Again, we argue in cases based on the possible configurations of curves in $S$ as listed in  Proposition \ref{prop:config}, and we will continue to use the notation established in Proposition \ref{prop:config}.

\smallskip
\emph{Case} (a). As in the proof of Proposition~\ref{prop:noII},
$N_{\mathcal F_j}\cdot C=a_0^{(j)}+a_1^{(j)}\le0<2$, so by Lemma~\ref{lem:ninv} the curve $C$ is
invariant for both foliations. The only effective divisors are $mC$, $m\ge1$, and so 
$[\Tang(\mathcal F_1, \mathcal F_2)] = [mC] = (-m,-m,0)$. Comparing with \eqref{eq:tang-class} we see that $(t_0, t_1, t_2) = (-m, -m, 0)$ for some $m \ge 1$, hence
\[
t_2=0\ \Longrightarrow\ A_2=-1;\qquad
t_0=t_1=-m\le-1\ \Longrightarrow\ A_0=A_1\le-2 .
\]
Since $A_i\ge-2$, this forces $A_0=A_1=-2$, i.e.\ $a_0^{(j)}=a_1^{(j)}=-1$ for $j=1,2$.
From $A_2=-1$, exactly one of $a_2^{(1)},a_2^{(2)}$ equals $0$. Assume, without loss of generality, that $a_2^{(1)}=0$. Then
\[
[N_{\mathcal F_1}]=(-1,-1,0)=[\mathcal{O}_S(C)].
\]
Since $C$ is $\FF_1$-invariant, Lemma~\ref{lem:kato-cr} applies with $E=C$ to imply that $S$ is Kato, which gives a
contradiction.

\smallskip
\emph{Case} (a$'$).
As above $N_{\mathcal F_j}\cdot C=a_0^{(j)}+a_1^{(j)}\le0<2$, so $C$ is invariant
for both $\mathcal F_1$ and $\mathcal F_2$. Similarly, since $[A]=(1,0,-1)$,
\[
N_{\mathcal F_j}\cdot A=-\big(a_0^{(j)}\cdot1+0+a_2^{(j)}(-1)\big)=-a_0^{(j)}+a_2^{(j)}\le1<2 ,
\]
so $A$ is invariant for both $\mathcal F_1$ and $\mathcal F_2$. By Lemma~\ref{lem:local}(i)--(ii), for each $j$ the
node $p$ of $C$ and the point $r=A\cap C$ (recall that by Proposition~\ref{prop:config}.(1), $r \neq p$)
both lie in $\Sing(\FF_j)$. Hence, by Lemma~\ref{lem:inv},
\[
Z(\FF_j,C)\ \ge\ 3
\ \Longrightarrow\
N_{\mathcal F_j}\cdot C=C^2-2+Z\ge-1
\ \Longrightarrow\
a_0^{(j)}+a_1^{(j)}\ge-1,
\]
\[
Z(\FF_j,A)\ \ge\ 1
\ \Longrightarrow\
N_{\mathcal F_j}\cdot A=A^2+Z\ge-1
\ \Longrightarrow\
-a_0^{(j)}+a_2^{(j)}\ge-1 .
\]
Using these two inequalities we have exactly six possibilities for $N_{\FF_j}$:
\[
(0,0,0),\ (0,0,-1),\ (0,-1,0),\ (0,-1,-1),\ (-1,0,0),\ (-1,0,-1)
\]
If $[N_{\mathcal F_j}] = (0, 0, 0)$ (resp. $ = (0, -1, -1)$) then $N_{\mathcal F_j} \equiv \mathcal O_S$ (resp. $\equiv \mathcal O_S(C+A)$). We may then apply Lemma~\ref{lem:kato-cr} with $E = 0$ (resp. $E = C+A$) to conclude that $S$ is Kato, which contradicts the hypotheses of the Proposition.

Thus, it remains only to consider when
$[N_{\mathcal F_j}]$ is one of four cases left:
\[
\{(0,0,-1),\,(0,-1,0),\,(-1,0,0),\,(-1,0,-1)\}.
\]
Effective divisors on $S$ are all of the form $mC+kA$ with $m,k\ge0$ and so
\[
[\Tang(\mathcal F_1, \mathcal F_2)] = [mC+kA] = m(-1,-1,0)+k(1,0,-1)=(-m+k,\,-m,\,-k)
\]
where $m, k \ge 0$.
Hence \eqref{eq:tang-class} gives $t_1=-m\le0$, so $A_1\le-1$, and hence either 
$a_1^{(1)}=-1$ or $a_1^{(2)}=-1$. Among the four cases above, only $(0,-1,0)$ satisfies this condition, and so without loss of generality we may assume that $[N_{\mathcal F_1}]=(0,-1,0)$. We consider in the following table all four possible classes for $[N_{\mathcal F_2}]$ and show that in each case we arrive at a contradiction.
\smallskip

\begin{center}
\small
\renewcommand{\arraystretch}{1.15}
\begin{tabular}{@{}llll@{}}
\toprule
$[N_{\mathcal F_2}]$ & $(A_0,A_1,A_2)$ & $(t_0,t_1,t_2)$ & Contradiction\\
\midrule
$(0,-1,0)$   & $(0,-2,0)$    & $(1,-1,1)$ & $t_2=-k=1\Rightarrow k=-1$, contradicting $k \ge 0$\\
$(0,0,-1)$   & $(0,-1,-1)$   & $(1,0,0)$  & $-m=0$, $-k=0$, but $-m+k=1$, contradiction.\\
$(-1,0,0)$   & $(-1,-1,0)$   & $(0,0,1)$  & $t_2=-k=1\Rightarrow k=-1$, contradicting $k \ge 0$\\
$(-1,0,-1)$  & $(-1,-1,-1)$  & $(0,0,0)$  & $m=k=0$, so $\Tang=0$, contradicting
Lemma~\ref{lem:tango}\\
\bottomrule
\end{tabular}
\end{center}
\smallskip

\noindent In every case we reach our required contradiction.

\smallskip
\emph{Case} (b). We have 
\[N_{\mathcal F_j}\cdot D_0=-a_1^{(j)}+a_2^{(j)}\le1<2 \quad \text{and}  \quad
N_{\mathcal F_j}\cdot D_1=a_0^{(j)}+a_1^{(j)}-a_2^{(j)}\le1<2,\]
so by Lemma~\ref{lem:ninv}
both curves $D_0$ and $D_1$ are invariant for both $\mathcal F_1$ and $\mathcal F_2$. By
Lemma~\ref{lem:local}(ii) both points of $D_0\cap D_1$ must be in $\Sing(\FF_j)$.

Then by
Lemma~\ref{lem:inv} 
\[
Z(\FF_j,D_i)\ \ge\ 2
\ \Longrightarrow\
N_{\mathcal F_j}\cdot D_i\ \ge\ D_i^2+2 ,
\]
that is,
\[
-a_1^{(j)}+a_2^{(j)}\ \ge\ D_0^2+2=0,
\qquad
a_0^{(j)}+a_1^{(j)}-a_2^{(j)}\ \ge\ D_1^2+2=-1 .
\]
These two inequalities show that $[N_{\FF_j}]$ can take on five possible values:  
\[
(0,0,0),\ (0,-1,0),\ (0,-1,-1),\ (-1,0,0),\ (-1,-1,-1)
\]
If $[N_{\mathcal F_j}] = (0, 0, 0)$ (resp. $ = (-1, 0, 0)$) we can apply Lemma~\ref{lem:kato-cr} with $E = 0$ (resp. $ = D_0+D_1$)
to conclude that $S$ is Kato, which contradicts the hypotheses of the Proposition.
Thus, $[N_{\mathcal F_j}]$ must fall into one of the following three cases:
\[
 \{(0,-1,0),\,(0,-1,-1),\,(-1,-1,-1)\}.
\]
Note that in each case we have $a_1=-1$, hence 
\begin{align}
 \label{eq:A1t1}   t_1=-1  
   \quad  \text{and} \quad A_1=-2.
\end{align}
Effective divisors on $S$ are all of the form $mD_0+kD_1$, $m,k\ge0$ not both zero, and so we have  
\[
[\Tang(\mathcal F_1, \mathcal F_2)] = [mD_0+kD_1] =  m(0,1,-1)+k(-1,-1,1)=(-k,\,m-k,\,-m+k),
\]
 where $m, k \ge 0$.
Note that the last two coordinates always sum to $0$, hence by comparing with \eqref{eq:tang-class} we must have $t_1+t_2=0$, which together with \eqref{eq:A1t1} implies that $t_2=1$ and $A_2=0$. Therefore, $a_2^{(1)}=a_2^{(2)}=0$ which, of our remaining possibilities for $[N_{\mathcal F_j}]$, is only satisfied by $(0,-1,0)$, hence $[N_{\mathcal F_1}]=[N_{\mathcal F_2}]=(0,-1,0)$. Thus,
\eqref{eq:tang-class} gives us 
\[(t_0, t_1, t_2) = (1, -1, 1) = (-k, m-k, -m+k) = [mD_0+kD_1]\]
which implies that $k = -1$, contradicting the fact that $k \ge 0$. 
\end{proof}

\subsection{Proof of main results and a conjecture}

\begin{proof}[Proof of Theorem~\ref{thm:main}]
Immediate from Proposition~\ref{prop:noIpairs}.  
(In fact, by Theorem~\ref{thm:count}(c) $S$ is an Inoue-Hirzebruch surface.)
\end{proof}

\begin{proof}[Proof of Corollary~\ref{cor:main}]
Immediate from Theorem~\ref{thm:main}.
\end{proof}

We finish the paper with the following conjecture

\begin{conjecture}\label{conj:n-1fol}
    Let $S$ be a class $\VIIp$ surface with $b_2(S)\ge 2$. Assume it possesses a pair of distinct singular holomorphic foliations. Then $S$ is a Kato surface.
\end{conjecture}


\

N. Kurnosov, \textit{London Institute for Mathematical Sciences, Royal Institution}, 21 Albemarle St, London W1S 4BS

e-mail:  \url{nk@lims.ac.uk}

\

C. Spicer, \textit{Kings College London}, Strand Building, Strand, London, WC2R 2LS

e-mail: \url{calum.spicer@kcl.ac.uk}


\begin{thebibliography}{99}

\bibitem{AD16} V.~Apostolov, G.~Dloussky,
\emph{Locally conformally symplectic structures on compact non-K\"ahler complex
surfaces}, Int.\ Math.\ Res.\ Not.\ IMRN \textbf{2016}, no.~9, 2717--2747.

\bibitem{AD18} V.~Apostolov, G.~Dloussky,
\emph{On the Lee classes of locally conformally symplectic complex surfaces},
J.~Symplectic Geom.\ \textbf{16} (2018), 931-958.

\bibitem{AD23} V.~Apostolov, G.~Dloussky,
\emph{Twisted differentials and Lee classes of locally conformally symplectic complex
surfaces}, Math.\ Z.\ \textbf{303} (2023), Paper No.~76; arXiv:2204.02122.

\bibitem{BFR23} G.~Barbaro, F.~Fagioli, \'A.~D.~R\'ios Ortiz,
\emph{A survey on rational curves on complex surfaces},
Rivista di Matematica della Universit\'a di Parma; arXiv:2209.04229.

\bibitem{BHPV} W.~Barth, K.~Hulek, C.~Peters, A.~Van de Ven,
\emph{Compact Complex Surfaces}, 2nd ed., Ergebnisse der Mathematik (3), vol.~4,
Springer, Berlin, 2004.

\bibitem{BB70} P.~Baum, R.~Bott,
\emph{On the zeroes of meromorphic vector fields},
in: Essays on Topology and Related Topics (M\'emoires d\'edi\'es \`a Georges de Rham),
Springer, New York, 1970, pp.~29-47.

\bibitem{Bog76} F.~A.~Bogomolov,
\emph{Classification of surfaces of class $\VII_0$ with $b_2=0$},
Izv.\ Akad.\ Nauk SSSR Ser.\ Mat.\ \textbf{40} (1976), 273-288;
Math.\ USSR Izv.\ \textbf{10} (1976), 255-269.

\bibitem{Bog82} F.~A.~Bogomolov,
\emph{Surfaces of class $\VII_0$ and affine geometry},
Izv.\ Akad.\ Nauk SSSR Ser.\ Mat.\ \textbf{46} (1982), 710-761;
Math.\ USSR Izv.\ \textbf{21} (1983), 31-73.

\bibitem{Bru97} M.~Brunella,
\emph{Feuilletages holomorphes sur les surfaces complexes compactes},
Ann.\ Sci.\ \'Ecole Norm.\ Sup.\ (4) \textbf{30} (1997), 569-594.

\bibitem{Bru00} M.~Brunella,
\emph{Birational Geometry of Foliations}, Monogr.\ Mat., IMPA, Rio de Janeiro, 2000.

\bibitem{Bru11} M.~Brunella,
\emph{On a class of foliated non-K\"ahlerian compact complex surfaces},
T\^ohoku Math.\ J.\ (2) \textbf{63} (2011), no.~3, 441-460.

\bibitem{Bru11LCK} M.~Brunella,
\emph{Locally conformally K\"ahler metrics on Kato surfaces},
Nagoya Math.\ J.\ \textbf{202} (2011), 77-81.

\bibitem{Bru13} M.~Brunella,
\emph{A characterization of Inoue surfaces},
Comment.\ Math.\ Helv.\ \textbf{88} (2013), 859-874; arXiv:1011.2035.

\bibitem{Bru14} M.~Brunella,
\emph{A characterization of hyperbolic Kato surfaces},
Publ.\ Mat.\ \textbf{58} (2014), no.~1, 251-261.



\bibitem{Buc00} N.~Buchdahl,
\emph{A Nakai-Moishezon criterion for non-K\"ahler surfaces},
Ann.\ Inst.\ Fourier (Grenoble) \textbf{50} (2000), 1533-1538.

\bibitem{BBK19}  F.~A.~Bogomolov, F.~Buonerba, N.~Kurnosov,
\emph{Classifying $\VII_0$ surfaces with $b_2=0$ via group theory}, arxiv: 1709.00062

\bibitem{CS82} C.~Camacho, P.~Sad,
\emph{Invariant varieties through singularities of holomorphic vector fields},
Ann.\ of Math.\ (2) \textbf{115} (1982), 579-595.

\bibitem{Dlo84} G.~Dloussky,
\emph{Structure des surfaces de Kato},
M\'em.\ Soc.\ Math.\ Fr.\ (N.S.) \textbf{14} (1984).

\bibitem{Dlo06} G.~Dloussky,
\emph{On surfaces of class $\VII_0^+$ with numerically anticanonical divisor},
Amer.\ J.~Math.\ \textbf{128} (2006), no.~3, 639-670.

\bibitem{Dlo21} G.~Dloussky,
\emph{Non-K\"ahlerian surfaces with a cycle of rational curves},
Complex Manifolds \textbf{8} (2021), 208-222.

\bibitem{Dlo24} G.~Dloussky,
\emph{On classification of compact complex surfaces of class $\VII$},
arXiv:2403.20178 (2024).

\bibitem{DK98} G.~Dloussky, F.~Kohler,
\emph{Classification of singular germs of mappings and deformations of compact surfaces
of class $\VII_0$}, Ann.\ Polon.\ Math.\ \textbf{70} (1998), 49-83.

\bibitem{DO99} G.~Dloussky, K.~Oeljeklaus,
\emph{Vector fields and foliations on compact surfaces of class $\VII_0$},
Ann.\ Inst.\ Fourier (Grenoble) \textbf{49} (1999), no.~5, 1503-1545.

\bibitem{DOT00} G.~Dloussky, K.~Oeljeklaus, M.~Toma,
\emph{Surfaces de la classe $\VII_0$ admettant un champ de vecteurs},
Comment.\ Math.\ Helv.\ \textbf{75} (2000), no.~2, 255-270.

\bibitem{DOT01} G.~Dloussky, K.~Oeljeklaus, M.~Toma,
\emph{Surfaces de la classe $\VII_0$ admettant un champ de vecteurs, II},
Comment.\ Math.\ Helv.\ \textbf{76} (2001), 640-664.

\bibitem{DOT03} G.~Dloussky, K.~Oeljeklaus, M.~Toma,
\emph{Class $\VII_0$ surfaces with $b_2$ curves},
T\^ohoku Math.\ J.\ (2) \textbf{55} (2003), 283-309.

\bibitem{DT12} G.~Dloussky, A.~Teleman,
\emph{Infinite bubbling in non-K\"ahlerian geometry},
Math.\ Ann.\ \textbf{353} (2012), 1283-1314.

\bibitem{DT20} G.~Dloussky, A.~Teleman,
\emph{Smooth deformations of singular contractions of class $\VII$ surfaces},
Math.\ Z.\ \textbf{296} (2020), 1521-1537.

\bibitem{Don87} S.~K.~Donaldson,
\emph{The orientation of Yang-Mills moduli spaces and $4$-manifold topology},
J.~Differential Geom.\ \textbf{26} (1987), 397-428.

\bibitem{Eno81} I.~Enoki,
\emph{Surfaces of class $\VII_0$ with curves},
T\^ohoku Math.\ J.\ (2) \textbf{33} (1981), 453-492.

\bibitem{Fav00} C.~Favre,
\emph{Classification of $2$-dimensional contracting rigid germs and Kato surfaces, I},
J.~Math.\ Pures Appl.\ \textbf{79} (2000), 475-514.

\bibitem{Ino74} M.~Inoue,
\emph{On surfaces of class $\VII_0$},
Invent.\ Math.\ \textbf{24} (1974), 269-310.

\bibitem{Kat78} Ma.~Kato,
\emph{Compact complex manifolds containing ``global'' spherical shells, I},
in: Proc.\ Int.\ Symp.\ Algebraic Geometry (Kyoto 1977), Kinokuniya, Tokyo, 1978,
pp.~45-84.

\bibitem{Kod64} K.~Kodaira,
\emph{On the structure of compact complex analytic surfaces, I},
Amer.\ J.~Math.\ \textbf{86} (1964), 751-798.

\bibitem{Kod66} K.~Kodaira,
\emph{On the structure of compact complex analytic surfaces, II},
Amer.\ J.~Math.\ \textbf{88} (1966), 682-721.

\bibitem{Kod68} K.~Kodaira,
\emph{On the structure of compact complex analytic surfaces, III},
Amer.\ J.~Math.\ \textbf{90} (1968), 55-83.



\bibitem{LYZ90} J.~Li, S.-T.~Yau, F.~Zheng,
\emph{A simple proof of Bogomolov's theorem on class $\VII_0$ surfaces with $b_2=0$},
Illinois J.~Math.\ \textbf{34} (1990), 217-220.

\bibitem{LYZ94} J.~Li, S.-T.~Yau, F.~Zheng,
\emph{On projectively flat Hermitian manifolds},
Comm.\ Anal.\ Geom.\ \textbf{2} (1994), 103-109.

\bibitem{LT95} M.~L\"ubke, A.~Teleman,
\emph{The Kobayashi-Hitchin Correspondence}, World Scientific, Singapore, 1995.

\bibitem{Nak84} I.~Nakamura,
\emph{On surfaces of class $\VII_0$ with curves},
Invent.\ Math.\ \textbf{78} (1984), 393-443.

\bibitem{Nak90} I.~Nakamura,
\emph{On surfaces of class $\VII_0$ with curves, II},
T\^ohoku Math.\ J.\ (2) \textbf{42} (1990), 475-516.

\bibitem{OT08} K.~Oeljeklaus, M.~Toma,
\emph{Logarithmic moduli spaces for surfaces of class $\VII$},
Math.\ Ann.\ \textbf{341} (2008), no.~2, 323-345.

\bibitem{OTZ01} K.~Oeljeklaus, M.~Toma, D.~Zaffran,
\emph{Une caract\'erisation des surfaces d'Inoue-Hirzebruch},
Ann.\ Inst.\ Fourier (Grenoble) \textbf{51} (2001), 1243-1257.

\bibitem{ST10} J.~Streets, G.~Tian,
\emph{A parabolic flow of pluriclosed metrics},
Int.\ Math.\ Res.\ Not.\ IMRN \textbf{2010}, no.~16, 3101-3133.

\bibitem{ST13} J.~Streets, G.~Tian,
\emph{Regularity results for pluriclosed flow},
Geom.\ Topol.\ \textbf{17} (2013), no.~4, 2389-2429.

\bibitem{Suw98} T.~Suwa,
\emph{Indices of Vector Fields and Residues of Singular Holomorphic Foliations},
Actualit\'es Math\'ematiques, Hermann, Paris, 1998.

\bibitem{Tel94} A.~Teleman,
\emph{Projectively flat surfaces and Bogomolov's theorem on class $\VII_0$ surfaces},
Internat.\ J.~Math.\ \textbf{5} (1994), no.~2, 253-264.

\bibitem{Tel05} A.~Teleman,
\emph{Donaldson theory on non-K\"ahlerian surfaces and class $\VII$ surfaces with
$b_2=1$}, Invent.\ Math.\ \textbf{162} (2005), 493-521.

\bibitem{Tel06} A.~Teleman,
\emph{The pseudo-effective cone of a non-K\"ahlerian surface and applications},
Math.\ Ann.\ \textbf{335} (2006), no.~4, 965-989.

\bibitem{TelGauge} A.~Teleman,
\emph{Gauge theoretical methods in the classification of non-K\"ahlerian surfaces},
arXiv:0804.0557 (2008).

\bibitem{Tel10} A.~Teleman,
\emph{Instantons and curves on class $\VII$ surfaces},
Ann.\ of Math.\ (2) \textbf{172} (2010), 1749-1804.

\bibitem{Tel15} A.~Teleman,
\emph{Instanton moduli spaces on non-K\"ahlerian surfaces. Holomorphic models around the
reduction loci}, J.~Geom.\ Phys.\ \textbf{91} (2015), 66-87.

\bibitem{Tel17det} A.~Teleman,
\emph{A variation formula for the determinant line bundle. Compact subspaces of moduli
spaces of stable bundles over class $\VII$ surfaces},
in: Geometry, Analysis and Probability, Progr.\ Math.\ 310, Birkh\"auser, 2017,
pp.~217-243.

\bibitem{Tel17} A.~Teleman,
\emph{Towards the classification of class $\VII$ surfaces},
in: Complex and Symplectic Geometry, Springer INdAM Ser.\ 21, Springer, Cham, 2017,
pp.~249-262.

\bibitem{Tel18} A.~Teleman,
\emph{Donaldson theory in non-K\"ahlerian geometry},
in: Modern Geometry: A Celebration of the Work of Simon Donaldson,
Proc.\ Sympos.\ Pure Math.\ 99, Amer.\ Math.\ Soc., 2018, pp.~363-392.

\bibitem{TW13} V.~Tosatti, B.~Weinkove,
\emph{The Chern-Ricci flow on complex surfaces},
Compos.\ Math.\ \textbf{149} (2013), no.~12, 2101-2138.

\end{thebibliography}
\end{document}